\documentclass{siamltex}
\usepackage{amsfonts}
\usepackage{amsmath}
\usepackage{color}
\usepackage{graphicx}

\newcommand{\RR}{{\mathbb{R}}}
\newcommand{\CC}{{\mathbb{C}}}

\newtheorem{remark}[theorem]{Remark}

\newtheorem{example}[theorem]{Example}

\begin{document}
\bibliographystyle{siam}

\pagestyle{myheadings}
\markboth{{\sc M.~Benzi and V.~Simoncini}}{Functions of matrices with Kronecker structure} 

\title {Approximation of functions of large matrices with Kronecker
structure}
\author{Michele Benzi\thanks{Department of Mathematics and Computer
Science, Emory University, Atlanta, Georgia 30322, USA
(benzi@mathcs.emory.edu). The work of this author was supported
by National Science Foundation grants 
DMS-1115692 and DMS-1418889.}  
\and 
Valeria Simoncini\thanks{Dipartimento di Matematica, Universit\`a di Bologna,
Piazza di Porta S.~Donato 5, I-40127 Bologna, Italy (valeria.simoncini@unibo.it).
The work of this author was partially supported by the FARB12SIMO grant
of the Universit\`a di Bologna.}
}

\maketitle

\begin{abstract}
We consider the numerical approximation of $f({\cal A})b$ where $b\in\RR^{N}$
and $\cal A$ is the sum of  Kronecker products, that is
${\cal A}=M_2 \otimes I + I \otimes M_1\in\RR^{N\times N}$. Here $f$ is a regular function
such that $f({\cal A})$ is well defined. We derive a computational strategy that
significantly lowers the memory requirements and computational efforts of the
standard approximations, with special emphasis on the exponential function,
for which the new procedure becomes
particularly advantageous. Our findings are illustrated by numerical experiments
with typical functions used in applications.
\end{abstract}


\section{Introduction}
We consider the problem of approximating
\begin{eqnarray}\label{eqn:main}
f({\cal A}) b
\end{eqnarray}
 where $f$ is a sufficiently
regular function defined on the spectrum of $\cal A$ (see \cite{Higham2008}), and
\begin{eqnarray}\label{eqn:kron}
{\cal A}=M_2 \otimes I + I \otimes M_1\in\RR^{N\times N}
\end{eqnarray}
 is the sum of Kronecker
products with $M_1 \in\RR^{n_1\times n_1}$, $M_2 \in\RR^{n_2\times n_2}$ so that
$N=n_1n_2$, $b={\rm vec}(B)$ with $B$ a matrix of low rank with dimensions compatible with 
that of $b$.
The Kronecker (or tensor)
product of two matrices $X$ and $Y$ of size $n_x\times m_x$ and $n_y\times m_y$,
respectively, is defined as
 $$
 X \otimes Y = 
\begin{bmatrix}
x_{11}Y & x_{12} Y & \cdots & x_{1 m_x} Y \\
x_{21}Y & x_{22} Y & \cdots & x_{2 m_x} Y \\
 \vdots & \vdots   & \ddots & \vdots      \\
x_{n_x 1}Y & a_{n_x 2} Y & \cdots & x_{n_x m_x} Y \\
\end{bmatrix} \in \RR^{n_x n_y \times m_x m_y};
 $$
the vec operator stacks the columns of a matrix $X=[x_1, \ldots, x_m]
 \in \RR^{n\times m}$
one after the other as
$$
{\rm vec}(X) = \begin{bmatrix} x_1 \\ \vdots \\ x_m\end{bmatrix} \in \RR^{nm\times 1} .
$$
The problem of approximating (\ref{eqn:main}) for general $\cal A$ is very important
in several applications and
has long attracted considerable attention;
we refer the reader to \cite{Higham2008} and
to \cite{Moler2003} for comprehensive treatments of the problem and for many
ways of numerically approximating its solution. For the case when $\cal A$ has large 
dimensions,
new effective approaches have been devised,
making the use of matrix function evaluations an important tool for solving
large scale (three-dimensional) scientific and engineering problems involving
discretized 
partial differential equations; see, e.g., 
\cite{Frommer2008b,Hochbruck1999,Hochbruck.Ostermann.10}.
In particular, the Kronecker structure above arises 
whenever the domain is a rectangle or a parallelepiped and finite difference or certain 
low-order finite element methods are employed to discretize differential equations with
separable coefficients; see, e.g., \cite{Druskin.Knizhnerman.98,DKZ09} and 
references therein. 
Other applications leading to matrices with Kronecker sum structure include
image processing \cite{HNO06}, queueing theory \cite[Chapter 9]{Ng2004},
graph analysis \cite[Chapter 3.4]{Bapat10}, and network design \cite{YX08}.

A significant body of literature is
now available on efficient numerical methods for approximately evaluating the product of 
$f({\cal A})$ times a vector $b$, using particular spectral properties  of $\cal A$
and under certain regularity conditions on $f$. To the best of our knowledge, 
the computational
advantages of exploiting, for a general function $f$, the possible
Kronecker structure of $\cal A$ have not been addressed in the 
context of Krylov subspace methods for large-scale
problems of the form (\ref{eqn:main}).
By taking into account this structure, and also the
possible low rank of $B$, the computational setting changes significantly. 
We will show that the memory requirements can be drastically reduced:
in fact, we show that by preserving the structure
of the problem, faster convergence and significantly lower memory requirements can be achieved.
More precisely, we acknowledge that the approximation to functions of $\cal A$
is the composition of distinct approximations in terms of $M_1$ and $M_2$, which
are much smaller matrices. Similar considerations can be made for other properties
of functions of matrices that are Kronecker sums, as is the case for their
sparsity and decay patterns; see \cite{Benzi.Simoncini.15tr} for a recent analysis.

Our results strongly rely on the low rank of the matrix $B$. In fact, but without
loss of generality, we shall assume that $B$ has rank equal to one, so that
we can write $B=b_1 b_2^T$, $b_1 \in\RR^{n_1}$, $b_2\in\RR^{n_2}$. For larger rank 
$\ell \ll \min\{n_1, n_2\}$, we could still write $B=B_1 B_2^T$ and proceed in
a similar manner. Our results also apply when $B$ is {\it numerically} low rank,
that is, only a few singular values of $B$ are above machine precision, or some other
small tolerance. In this case, we could write $b=\widehat b + b_\epsilon$ 
where $\widehat b = {\rm vec}(\widehat B)$ with $\widehat B$ of low rank, 
and $\|b_\epsilon\|\ll 1$. If $\|f({\cal A})\|$ is not too large,
$$
f({\cal A}) b = f({\cal A}) \widehat b + f({\cal A}) b_\epsilon \approx
 f({\cal A}) \widehat b .
$$

An outline of the paper is as follows.
In section~\ref{sec:proj} we review some standard techniques for approximating
(\ref{eqn:main}) when $\cal A$ is large, and set up the notation for the rest
of the paper. In section \ref{sec:kron} we derive the structure-exploiting approximation
for general functions $f$ such that $f(\cal A)$ is well defined. In section \ref{sec:exp}
we focus on the exponential function, for which the new procedure becomes
particularly advantageous. 
Another important special case, the matrix inverse, is briefly discussed in section 
\ref{sec:inv}, and more general matrix functions in section \ref{sec:scm}. 
Conclusions are given in section \ref{sec:con}.
Our findings are illustrated by numerical experiments
with typical functions used in applications.

\section{General approximation by projection}\label{sec:proj}
A common procedure for large ${\cal A}$
constructs an approximation space, and a matrix ${\cal V}$ whose
orthonormal columns span that space, and obtain
\begin{eqnarray}\label{eqn:proj}
f({\cal A}) b \approx {\cal V} f(H) e, 
\qquad H={\cal V}^T{\cal A}{\cal V}, \quad e={\cal V}^Tb .
\end{eqnarray}
Depending on the spectral properties of the matrix 
${\cal A}$ and on the vector $b$, the approximation space dimension may need to be
very large to obtain a good approximation. Unfortunately, the whole matrix
$\cal V$ may need to be stored, limiting the applicability of the approach.
This is the motivation behind the recently introduced restarted methods, which try
to cope with the growing space dimensions by restarting the approximation process
as soon as a fixed maximum subspace dimension is reached 
\cite{Eiermann2006,Frommeretal.14}.

A classical choice as approximation space is given by the (standard) Krylov subspace\footnote{In case
$b$ is a matrix, the definition of a ``block'' Krylov subspace is completely analogous, that is
 $K_m({\cal A}, b)  = {\rm range}([b, {\cal A}b, \ldots, {\cal A}^{m-1}b])$.}
 $K_m({\cal A}, b)  = {\rm span}\{b, {\cal A}b, \ldots, {\cal A}^{m-1}b\}$.
An orthonormal basis $\{v_1, \ldots, v_m\}$ can be constructed sequentially
via the Arnoldi recurrence, which can be written in short as
$$
{\cal A} {\cal V}_m = {\cal V}_m H_m + h v_{m+1}e_m^T, \qquad {\cal V}_m=[v_1, \ldots, v_m];
$$
here $e_m$ is the $m$th vector of the canonical basis of $\RR^m$, $H_m={\cal V}_m^T 
{\cal A}{\cal V}_m$ (as
stated earlier), and $h=\|{\cal A}v_m - \sum_{i=1}^m[H_m]_{im}v_i\|$.

The past few years have seen a rapid increase in the use of richer approximation spaces
than standard Krylov subspaces. More precisely, rational Krylov subspaces, namely
\begin{eqnarray}\label{eqn:rks}
{\mathbb K}_m({\cal A}, b, {\pmb \sigma}_{m-1}) = 
{\rm span}\{b, ({\cal A}-\sigma_1 I)^{-1}b, \ldots, \prod_{i=1}^{m-1} ({\cal A}-\sigma_i I)^{-1} b\},
\end{eqnarray}
have been shown to be particularly well suited for matrix function approximations; we
refer the reader to \cite{Guettel.survey.13} for a recent survey on various issues
related to rational Krylov subspace approximations of matrix functions.
A special case is given by the extended Krylov subspace, which alternates powers
of ${\cal A}$ with powers of ${\cal A}^{-1}$ 
\cite{Druskin.Knizhnerman.98,KnizhnermanSimoncini2010}.

\section{Exploiting the Kronecker structure}\label{sec:kron}
Assume that $\cal A$ has the form in (\ref{eqn:kron}) and that, for simplicity,
$B$ has rank one, that is $B=b_1 b_2^T$. We generate distinct approximations for
the matrices $M_1$ and $M_2$; in the case of the classical Krylov subspace these
are given as
$$
K_m(M_1,b_1), \quad M_1 Q_m = Q_m T_1 + q_{m+1} t^{(1)} e_m^T
$$
and 
$$
K_m(M_2,b_2), \quad M_2 P_m = P_m T_2 + p_{m+1} t^{(2)} e_m^T .
$$
Note that the two spaces could have different dimensions; we will use
the same dimension for simplicity of presentation.
The matrices $Q_m$ and $P_m$ have orthonormal columns, and
 have a much smaller number of rows than ${\cal V}_m$ (the square root of it,
if $n_1=n_2$).
We thus consider the following quantity to define the approximation:
$$
{\cal V} = P_m \otimes Q_m
$$
so that
\begin{eqnarray*}
{\cal A} {\cal V} &=& {\cal A} (P_m \otimes Q_m) = M_2 P_m \otimes Q_m + P_m \otimes M_1 Q_m  \\
&=& (P_m T_2 \otimes Q_m + P_m \otimes Q_m T_1) + p_{m+1} t^{(2)} e_m^T
 \otimes Q_m + P_m \otimes q_{m+1} t^{(1)} e_m^T \\
&=& (P_m \otimes Q_m) (T_2 \otimes I_m + I_m \otimes T_1) + {\rm low\,\, rank}  .
\end{eqnarray*}
Following the general setting in (\ref{eqn:proj}), and
defining ${\cal T}_m = T_2\otimes I_m + I_m \otimes T_1$
 we thus consider the approximation
\begin{eqnarray} \label{eqn:newapprox}
f({\cal A}) b \,\, \approx \,\,
x_m^{\otimes} :=
(P_m \otimes Q_m) z, \quad z= f({\cal T}_m)   (P_m \otimes Q_m)^Tb  .
\end{eqnarray}
We stress that the matrix $P_m \otimes Q_m$ does not need to be explicitly
computed and stored. Indeed, 
letting $Z\in\RR^{m\times m}$ be such that $z={\rm vec}(Z)$, it holds that 
$x_m^{\otimes} ={\rm vec}(Q_m Z P_m^T)$; moreover,
$(P_m \otimes Q_m)^Tb = {\rm vec}((Q_m^Tb_1)(b_2^TP_m))$. 
The following proposition provides a
cheaper computation in case both $T_1$ and $T_2$ are diagonalizable, as is the case
for instance when they are both symmetric.

\begin{proposition}\label{prop:comput}
Assume that the matrices $T_1, T_2$ are diagonalizable, and let
$T_1=X\Lambda X^{-1}$,
$T_2=Y\Theta Y^{-1}$ be their eigendecompositions. 
Let 
$$
g=f(\Theta\otimes I_m + I_m\otimes \Lambda) {\rm vec}(X^{-1} Q_m^Tb_1 b_2^T P_m Y^{-T}) \in\RR^{m^2}.
$$
With the notation and assumptions above, for $G$ such that
$g={\rm vec}(G)$ it holds that
$$
x_m^{\otimes}= {\rm vec}(Q_m X G Y^T P_m^T) .
$$
\end{proposition}

\begin{proof}
Using the properties of the Kronecker product
(see, e.g., \cite[Corollary 4.2.11 and Theorem 4.4.5]{Horn.Johnson.91}),
the eigendecomposition of ${\cal T}_m=T_2\otimes I + I \otimes T_1$ is given by
$$
T_2\otimes I + I \otimes T_1 = (Y\otimes X) (\Theta\otimes I_m + I_m \otimes \Lambda) 
(Y^{-1}\otimes X^{-1}) ,
$$
so that
$f({\cal T}_m) = (Y\otimes X) f(\Theta\otimes I_m + I_m \otimes \Lambda)
(Y^{-1}\otimes X^{-1})$, where $f(\Theta\otimes I_m + I_m \otimes \Lambda)$ is
a diagonal matrix. The result follows from explicitly writing down the
eigenvector matrices associated with each Kronecker product; note that
 $f(\Theta\otimes I_m + I_m \otimes \Lambda)$ can be computed cheaply as both
$\Theta$ and $\Lambda$ are diagonal.
\end{proof}

\vspace{0.1in}

In addition to providing a computational procedure for determining
$x_m^\otimes$,
Proposition \ref{prop:comput} reveals that, in exact arithmetic,
the true vector $x = f({\cal A})b$ can be obtained using information
from spaces of dimension
at most $m=\max\{n_1, n_2\}$, whereas the standard approximation $x_m$
may require a much larger dimension space. This fact is due to
both the Kronecker form of $\cal A$ and 
the composition of $b$, as $b$ corresponds to the ``vectorization''
of the rank-one matrix $b_1 b_2^T$.  
The following examples 
illustrate this property, while more explicit formulas can be
obtained for the exponential function, as we will describe in
section~\ref{sec:exp}.

\vskip 0.1in
\begin{example}\label{example:sqrt}
{\rm We consider $f(x)=\sqrt{x}$ and
$M_1=M_2 = -{\rm tridiag}(1,-2,1)\in\RR^{n\times n}$, $n=50$, each corresponding to
the (scaled) centered three-point discretization of the one-dimensional 
negative Laplace operator in (0,1).  We first consider $b_1$ equal to the vector
of all ones, and
$b_2$ a vector of random values uniformly distributed in $(0,1)$; 
the results are shown in Table~\ref{tab:sqrt}.
We observe that convergence is faster, in terms of space dimension, for $x_m^{\otimes}$.
  Moreover, once subspaces
of dimension $m=n=50$ are reached, a rather accurate approximation is obtained
with the structure-preserving approach, as the full eigenspace of $M_1$
is generated.
We next consider the case of $b_2=b_1$ (the vector of all ones) and 
report the numerical experiments
in Table~\ref{tab:sqrt2}. Convergence is even faster in the structure preserving
method, as apparently convergence is faster with $b_1$ than with the original $b_2$. 
No major difference is observed in the standard procedure.
}
\end{example}
\vskip 0.1in

{\footnotesize
\begin{figure}
\begin{minipage}[c]{.47\textwidth}
\centering
\begin{tabular}{|rrr|}
$m$ & $\|f({\cal A})b - x_m\|$ & $\|f({\cal A})b - x_m^{\otimes}\|$  \\ \hline
    5 & 1.4416e+00 & 9.6899e-01 \\
   10 & 5.2832e-01 & 2.7151e-01 \\
   15 & 2.2517e-01 & 8.4288e-02 \\
   20 & 9.9517e-02 & 1.8327e-02 \\
   25 & 4.0681e-02 & 8.5632e-03 \\
   30 & 1.5114e-02 & 2.7162e-03 \\
   35 & 9.0086e-03 & 5.3891e-04 \\
   40 & 6.3515e-03 & 1.9269e-04 \\
   45 & 3.5355e-03 & 1.9476e-05 \\
   50 & 1.7627e-03 & 6.4440e-13 \\ 
\hline
\end{tabular}
\caption{Example \ref{example:sqrt} for $f(x)=\sqrt{x}$. $b_2\ne b_1$.\label{tab:sqrt}}
\end{minipage}
\hspace{6mm}
\begin{minipage}[c]{.47\textwidth}
\begin{tabular}{|rrr|}
$m$ & $\|f({\cal A})b - x_m\|$ & $\|f({\cal A})b - x_m^{\otimes}\|$  \\ \hline
    5 & 1.9371e+00 & 1.5903e+00 \\
   10 & 7.5344e-01 & 4.5636e-01 \\
   15 & 3.3417e-01 & 1.3538e-01 \\
   20 & 1.4240e-01 & 2.5706e-02 \\
   25 & 5.1205e-02 & 1.1719e-12 \\
   30 & 1.2671e-02 & 1.1034e-12 \\
   35 & 5.1316e-03 & 1.4357e-12 \\
   40 & 1.7854e-03 & 1.1186e-12 \\
   45 & 6.2249e-04 & 1.2297e-12 \\
   50 & 1.8720e-04 & 1.2975e-12 \\
\hline
\end{tabular}
\caption{Example \ref{example:sqrt} for $f(x)=\sqrt{x}$. $b_2=b_1$.\label{tab:sqrt2}}
\end{minipage}
\end{figure}
}

\begin{example}\label{example:ex_Frommer_1}
{\rm
Data for this example are taken from \cite{Frommeretal.14}.
We consider the function $f(z) = (e^{s\sqrt{z}}-1)/z$ with $s=10^{-3}$, and
$A=M\otimes I + I\otimes M$, where $M$ is the $n\times n$ tridiagonal matrix of
the finite difference discretization of the one-dimensional Laplace operator;
$b:=b_1=b_2$ is the vector of all ones.
Table \ref{tab:ex_Frommer1} shows the approximation history of the
standard method and of the new approach for $n=50$. Because of the
small size, we could compute and monitor the true error. In the last two
columns, however, we also report the relative difference between the
last two approximation iterates, which may be considered as a simple-minded
 stopping criterion for
larger $n$; see, e.g., 
\cite{KnizhnermanSimoncini2010} or \cite{FrommerGuettelSchweitzer.14} for more 
sophisticated criteria. The results are as those of the previous examples.
In Table \ref{tab:ex_Frommer2} we report the runs for $n=100$, for which
we could not compute the exact solution, so that only the error estimates are
reported. The results are very similar to the smaller case. In this case,
memory requirements of the structured approximation become significantly
lower than for the standard approach.
}
\end{example}

\begin{table}
\centering
\begin{tabular}{|rrrrr|}
$m$ & $\|f({\cal A})b - x_m\|$ & $\|f({\cal A})b - x_m^{\otimes}\|$ &
$\frac{\|x_{m} - x_{m,old}\|}{\|x_{m}\|}$ & 
$\frac{\|x_{m}^{\otimes} - x_{m,old}^{\otimes}\|}{\|x_{m}^{\otimes}\|}$ \\ \hline
    4 & 4.2422e-01 & 3.9723e-01 & 1.0000e+00 & 1.0000e+00 \\
    8 & 2.6959e-01 & 2.1025e-01 & 2.2710e-01 & 2.5313e-01 \\
   12 & 1.7072e-01 & 1.0365e-01 & 1.3066e-01 & 1.2971e-01 \\
   16 & 1.0324e-01 & 4.2407e-02 & 8.3444e-02 & 6.9960e-02 \\
   20 & 5.7342e-02 & 1.1176e-02 & 5.4224e-02 & 3.3969e-02 \\
   24 & 2.7550e-02 & 4.8230e-04 & 3.4054e-02 & 1.0935e-02 \\
   28 & 1.0351e-02 & 2.8883e-12 & 1.9296e-02 & 4.8230e-04 \\
   32 & 3.4273e-03 & 2.8496e-12 & 8.3585e-03 & 1.1366e-13 \\
   36 & 2.2906e-03 & 2.9006e-12 & 1.7514e-03 & 1.4799e-13 \\
   40 & 9.4368e-04 & 2.8119e-12 & 1.6283e-03 & 2.7323e-13 \\
   44 & 4.3935e-04 & 2.7593e-12 & 6.2797e-04 & 2.1786e-13 \\
   48 & 1.8744e-04 & 2.8235e-12 & 3.0332e-04 & 2.5965e-13 \\
\hline
\end{tabular}
\caption{Example \ref{example:ex_Frommer_1}.
For $f(x)= (e^{s\sqrt{z}}-1)/z$ with $s=10^{-3}$. Here $n=50$.\label{tab:ex_Frommer1}}
\end{table}

\begin{table}
\centering
\begin{tabular}{|rrr|}
$m$ & 
$\frac{\|x_{m} - x_{m,old}\|}{\|x_{m}\|}$ & 
$\frac{\|x_{m}^{\otimes} - x_{m,old}^{\otimes}\|}{\|x_{m}^{\otimes}\|}$ \\ \hline
    4 &  1.0000e+00 & 1.0000e+00 \\
    8 &  2.3942e-01 &  2.7720e-01 \\
   12 &  1.5010e-01 &  1.6289e-01 \\
   16 &  1.0716e-01 &  1.0966e-01 \\
   20 &  8.1062e-02 &  7.8150e-02 \\
   24 &  6.3308e-02 &  5.7003e-02 \\
   28 &  5.0347e-02 &  4.1674e-02 \\
   32 &  4.0409e-02 &  2.9992e-02 \\
   36 &  3.2507e-02 &  2.0802e-02 \\
   40 &  2.6052e-02 &  1.3446e-02 \\
   44 &  2.0667e-02 &  7.5529e-03 \\
   48 &  1.6104e-02 &  2.9970e-03 \\
   52 &  1.2194e-02 &  3.1470e-04 \\
   56 &  8.8234e-03 &  1.1354e-12 \\
   60 &  5.9194e-03 &  3.4639e-13 \\
\hline
\end{tabular}
\caption{Example \ref{example:ex_Frommer_1}.
For $f(x)= (e^{s\sqrt{z}}-1)/z$ with $s=10^{-3}$ and $b_2=b_1$. Here $n=100$.\label{tab:ex_Frommer2}}
\end{table}

\section{The case of the matrix exponential}\label{sec:exp}
The evaluation of (\ref{eqn:main}) with $f({\cal A}) = \exp({\cal A})$ 
presents special interest owing to its importance in the numerical
solution of time-dependent ODEs and PDEs 
\cite{Hochbruck1997,Hochbruck1999,Hochbruck.Ostermann.10}.
The problem also arises in network science, when evaluating
the {\em total communicability} of a network \cite{Benzi2013,EHB12}. 

The exponential function provides a particularly favorable setting in the case
of a matrix having Kronecker form.  Indeed, due to the
property (see, e.g., \cite[Theorem 10.9]{Higham2008})
\begin{eqnarray}\label{eqn:exp_kron}
\exp(T_2 \otimes I_m + I_m \otimes T_1) =
\exp(T_2)\otimes \exp(T_1) ,
\end{eqnarray}
formula (\ref{eqn:newapprox}) simplifies even further.
Indeed, we obtain
\begin{eqnarray}
x_m^{\otimes}&=&
(P_m \otimes Q_m) \exp(T_2 \otimes I_m + I_m \otimes T_1)(P_m \otimes Q_m)^Tb \nonumber\\
\qquad \qquad &=&
(P_m \otimes Q_m) (\exp(T_2) \otimes \exp(T_1)) (P_m \otimes Q_m)^Tb\nonumber \\
\qquad \qquad &=&
 {\rm vec}(Q_m \exp(T_1) Q_m^T b_1 b_2^T P_m\exp(T_2)^T P_m^T)\nonumber \\
\qquad \qquad 
&=& 
{\rm vec}(x_m^{(1)} (x_m^{(2)})^T),  \label{eqn:xotimes}
\end{eqnarray}
 with $x_m^{(1)}=Q_m \exp(T_1) Q_m^T b_1$ and $x_m^{(2)} =P_m \exp(T_2) (P_m^Tb_2)$.
We observe that the final approximation
$x_m^{\otimes}$ is the simple combination of the two separate approximations
of $\exp(M_1) b_1$ and $\exp(M_2) b_2$.
Indeed, the same approximation could be obtained by first writing
\begin{eqnarray}\label{eqn:kron_form}
\exp({\cal A}) b = \exp(M_2)\otimes \exp(M_1) b = {\rm vec}( (\exp(M_1)b_1) (b_2^T\exp(M_2)^T),
\end{eqnarray}
and then using the standard approximations
$\exp(M_1)b_1\approx Q_m \exp(T_1) Q_m^T b_1$ and
 $\exp(M_2)b_2 \approx P_m \exp(T_2) P_m^T b_2$. 

In the following we illustrate the behavior of the approximation to the
matrix exponential with a few numerical examples. Here the standard Krylov subspace
is used in all instances for approximating the corresponding vector. We stress
that because of the decreased memory allocations to generate $K_m(M_i, b_i)$, 
the computation of $x_m^{\otimes}$
can afford significantly larger values of $m$, than when building
$K_m({\cal A}, b)$.

\begin{figure}[tb]
\centering
\includegraphics[width=2.5in,height=2.5in]{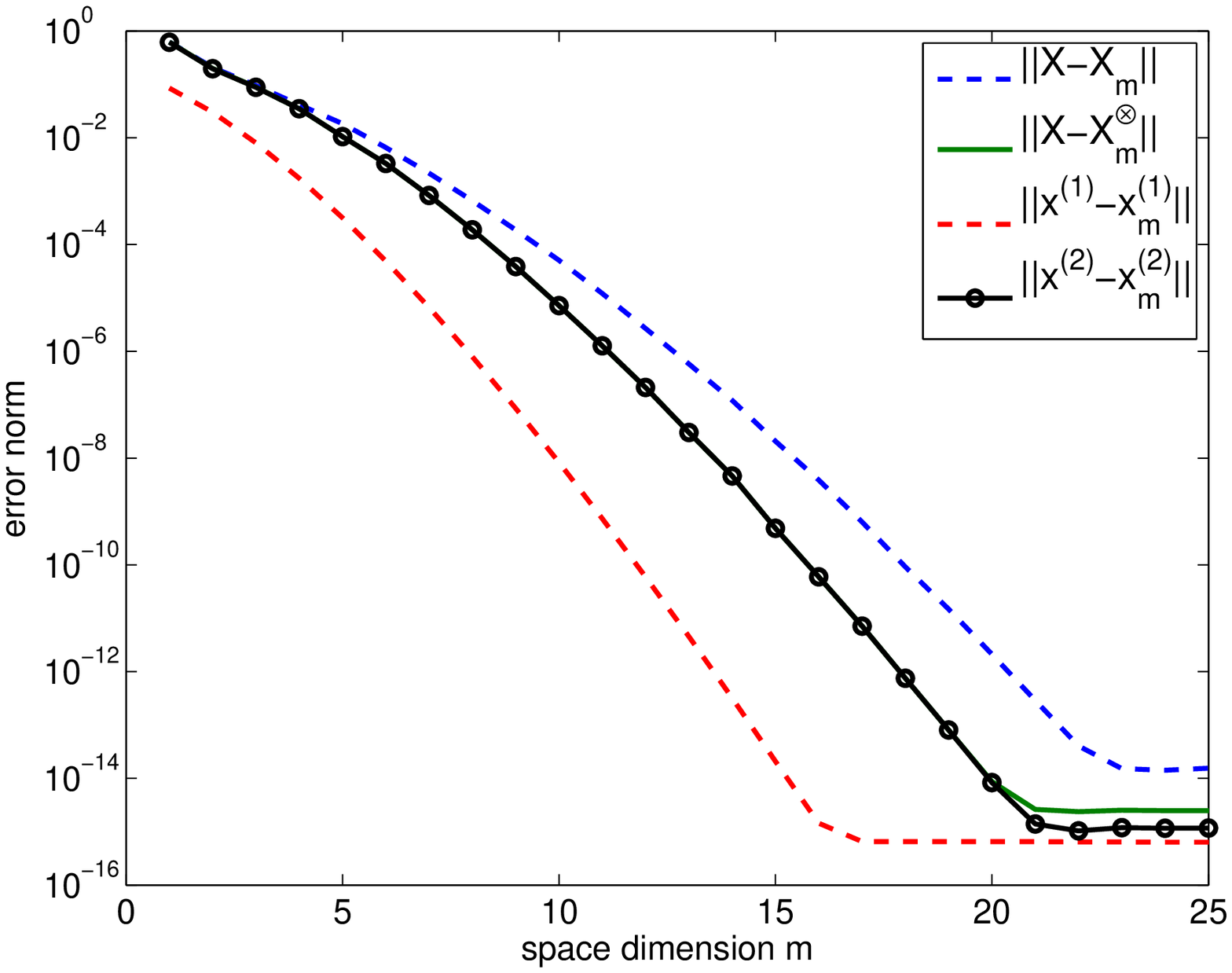}
\includegraphics[width=2.5in,height=2.5in]{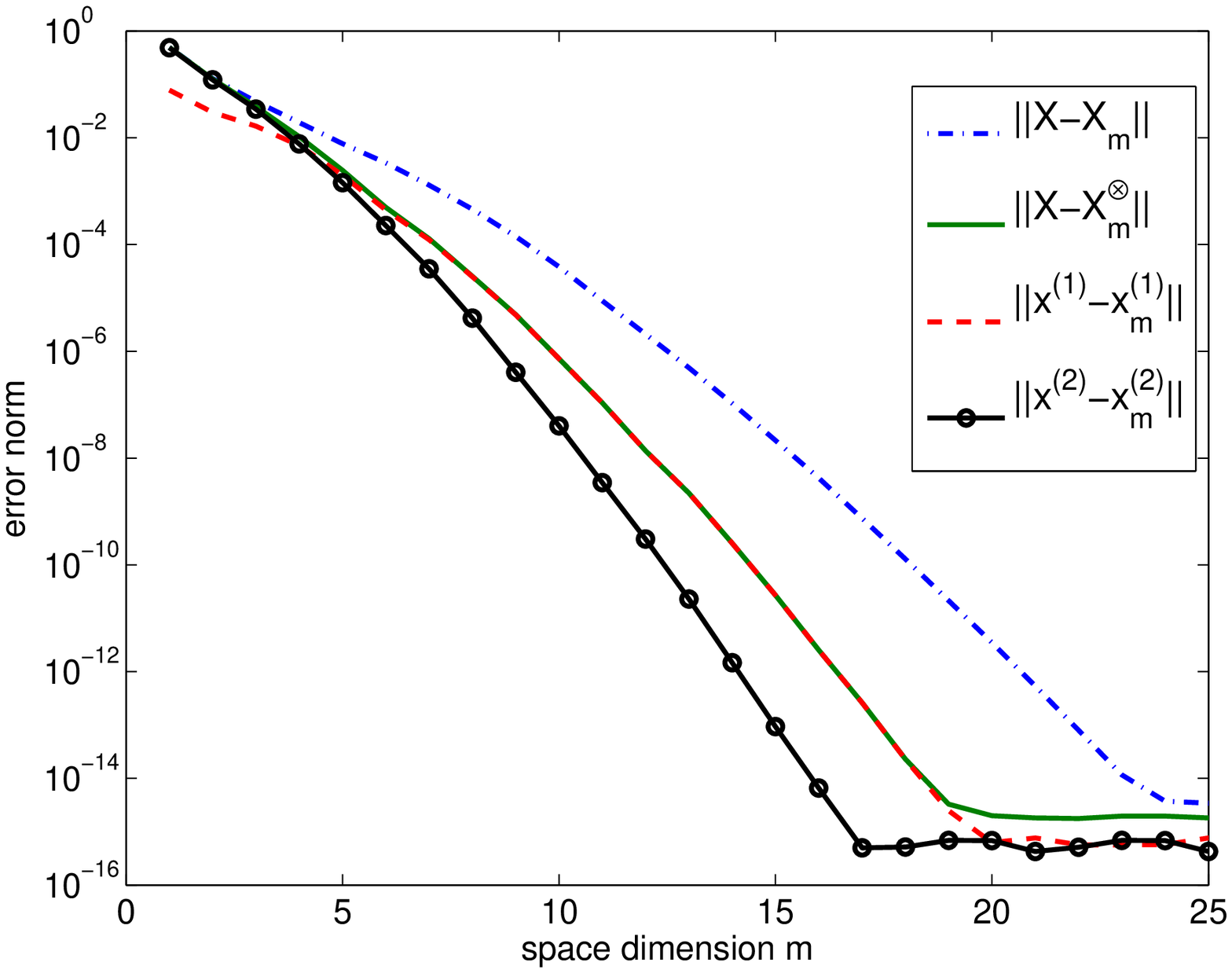}
\caption{Convergence history to $\exp({\cal A})b$. Left: Example \ref{ex:exp}. Right: Example \ref{ex:exp_nonsym}. \label{fig:exp}}
\end{figure}

\begin{example}\label{ex:exp}
{\rm
We consider the approximation of $\exp({\cal A})b$ with $\cal A$ as in (\ref{eqn:kron}) and
$M_1={\rm tridiag}{(1,-2,1)}$ and $M_2={\rm tridiag}{(2,-3,2)}$, both  of size $n_1=n_2=70$.
Therefore, $\cal A$ has dimension $4900$.
Moreover, we take $b_1=[1, \ldots, 1]^T$ and $b_2$ a vector with random values uniformly
distributed in $(0,1)$. Thanks to the problem size, the vector
$\exp({\cal A})b$ could be computed explicitly.
Figure \ref{fig:exp}(left) reports  the convergence history as the space dimension
$m$ increases when using two different approaches:  the first one uses 
$K_m({\cal A},b)$ as approximation space, so that
$x_m={\cal V}_m \exp(H_m)({\cal V}_m^T b)$ with $H_m={\cal V}_m^T{\cal A}{\cal V}_m$; 
the second one uses $x_m^{\otimes}$ in (\ref{eqn:xotimes}).
We observe that the convergence of $x_m^{\otimes}$ is faster than that of $x_m$; this
fact will be explored in section \ref{sec:conv}.
The plot also reports the error norm in the approximation of $x^{(1)}$ and $x^{(2)}$:
the error norm
for $x_m^{\otimes}$ is mainly driven by that of the most slowly converging approximation
between $x_m^{(1)}$ and $x_m^{(2)}$.
}
\end{example}

\begin{example}\label{ex:exp_nonsym}
{\rm
We modify Example \ref{ex:exp} by setting 
$M_2={\rm tridiag}{(1,-2,1)}$, while
$M_1$ to be equal to the discretization by finite differences of the
one-dimensional non-selfadjoint operator ${\cal L}(u) = u_{xx} - 100 u_x$ 
on the interval $[0,1]$. Hence, $M_1$ (and therefore $\cal A$) is
nonsymmetric.
The matrix dimensions and the vectors $b_1, b_2$ are as in the previous example.
The convergence history is reported in the right plot of Figure \ref{fig:exp}. Similar
comments as for Example \ref{ex:exp} can be deduced.
}
\end{example}


\begin{figure}[htb]
\centering
\includegraphics[width=2.5in,height=2.5in]{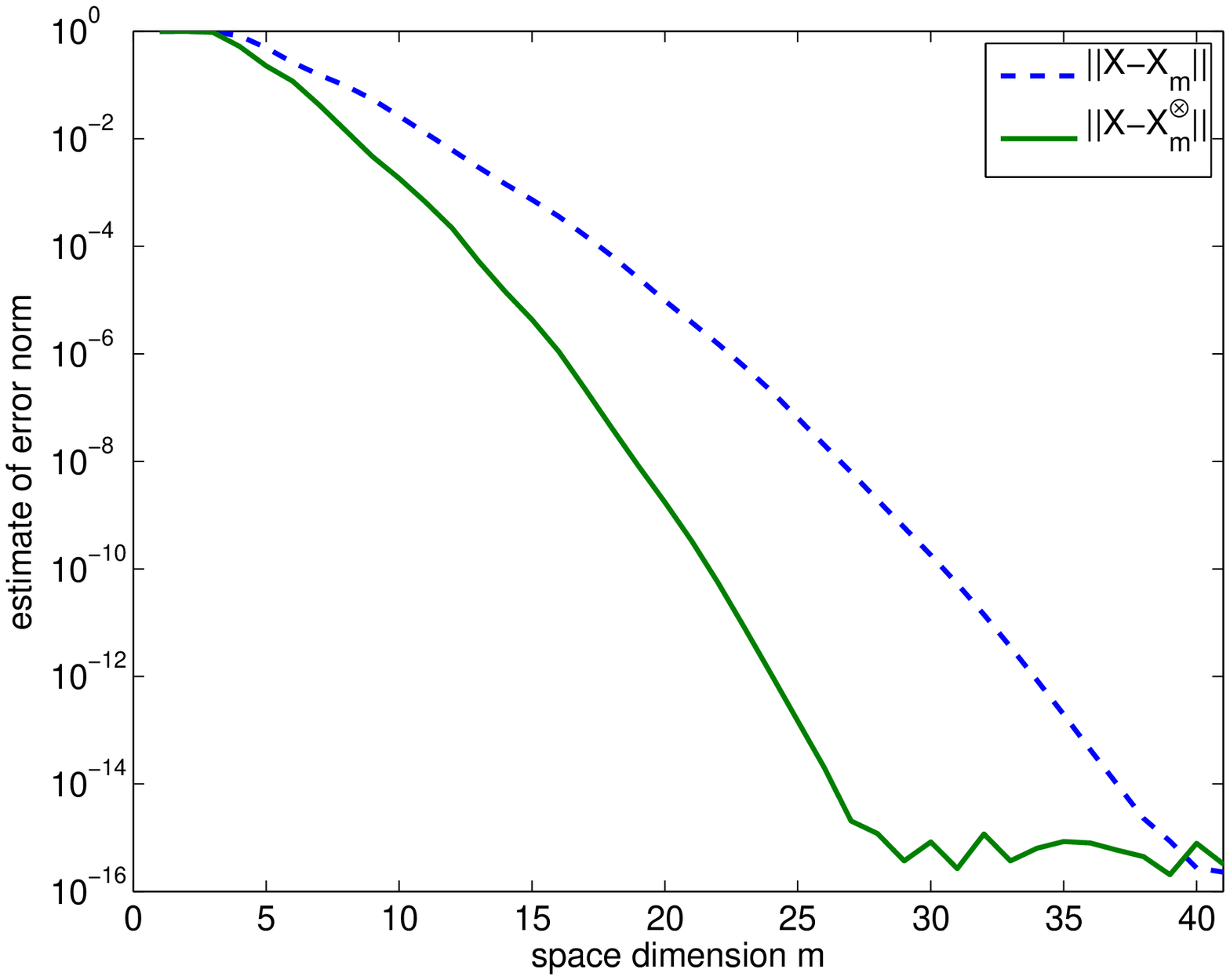}
\includegraphics[width=2.5in,height=2.5in]{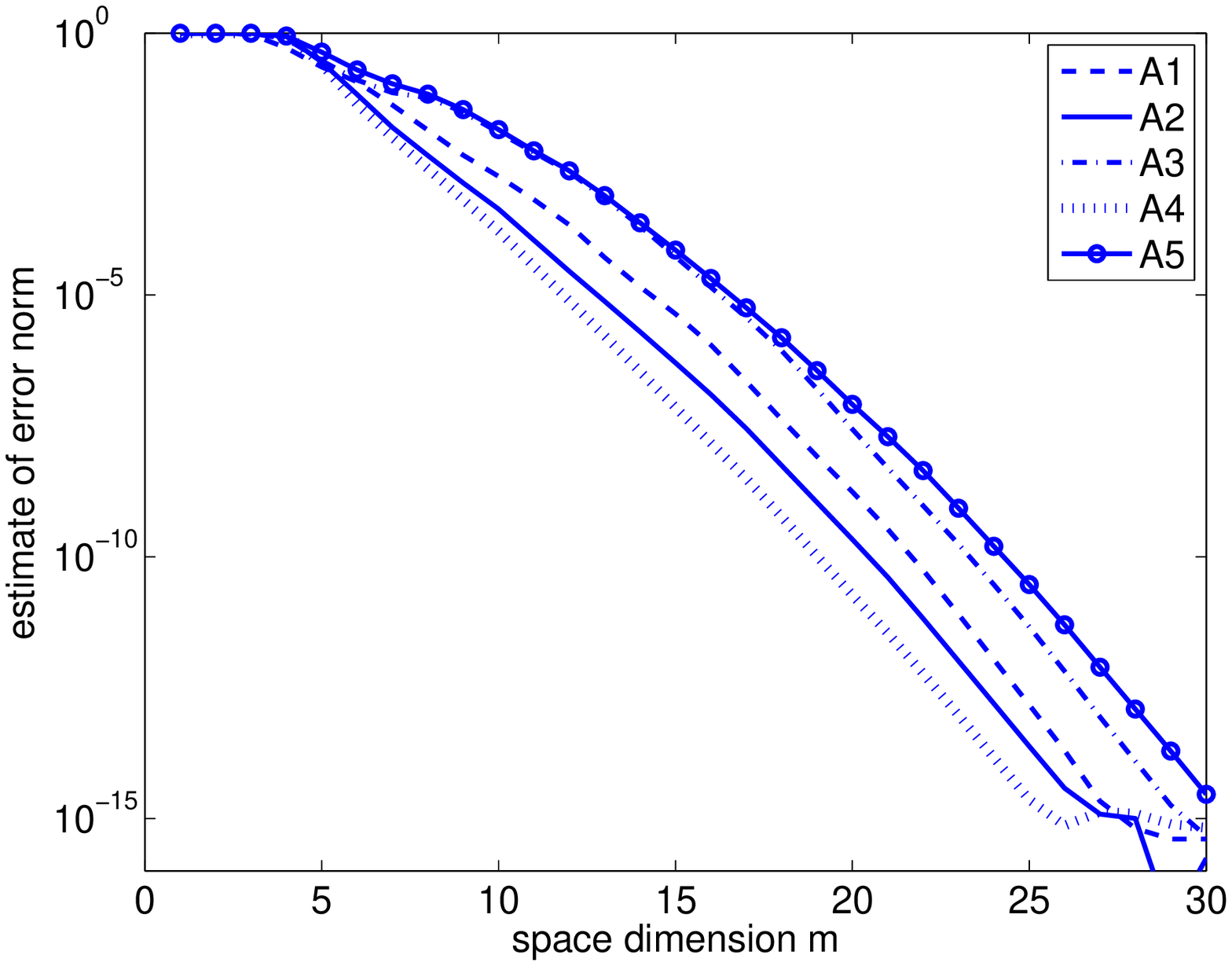}
\caption{Example \ref{ex:graphs}. Convergence history to $\exp({\cal A})b$. 
Left: Example \ref{ex:graphs}, case $n=1000$.  Right: Example \ref{ex:graphs},
convergence history for all five cases.  \label{fig:graphs}}
\end{figure}

\begin{example}\label{ex:graphs}
{\rm The next example arises in graph and network analysis. Given two graphs
$G_1 = (V_1, E_1)$ and $G_2 = (V_2, E_2)$,
we consider the Cartesian product ${\cal G} = G_1\square G_2$ of the two
given graphs, defined as follows. The vertex set of $\cal G$ is just the
Cartesian product $V_1\times V_2$, and there is an edge between two
vertices $(u_1,u_2)$ and $(v_1,v_2)$ of $\cal G$ if either $u_1 = v_1$
and $(u_2,v_2)\in E_2$, or $u_2 = v_2$ and $(u_1,v_1)\in E_1$.
The adjacency matrix of $\cal G$ is then the Kronecker sum of the
adjacency matrices of $G_1$ and $G_2$ \cite[page 37]{Bapat10}; see also 
\cite{YX08} for definitions (and applications) in the case of
directed graphs.
A useful notion in the
analysis of complex networks is the {\em total communicability}, which
is defined as the row sum of the exponential of the adjacency matrix,
see \cite{Benzi2013}. The entries of this vector provide a measure
of the ``importance" of the nodes in the network, and can be computed
as $\exp({\cal A})b$ where now $b$ is the
vector of all ones (note that the corresponding matrix $B$ has rank one).
We consider five Cartesian product graphs of the form
${\cal G}_i = G_i\square G_i$, with each $G_i$ being a Barabasi--Albert
graph constructed using the preferential attachment model. The
command {\tt pref} in the Matlab toolbox {\sc contest} \cite{contest09} was used
(with the default choice of parameters) to generate five graphs on $n$
nodes, where $n=1000,2000,\ldots ,5000$. Thus, the adjacency matrices
of the corresponding Cartesian product graphs ${\cal G}_i$ have dimension
ranging between one and twenty-five millions. All the resulting matrices
are symmetric indefinite.

 Table \ref{tab:graphs} reports the CPU time required to compute a basis for the
Krylov subspace of dimension $m=30$ as the graph matrix size increases (all runs
were performed with Matlab R2011b \cite{matlab7} on a laptop with  Intel Core i7-3687U CPU 
running at 2.10Ghz with 7.7GiB memory).
The last column reports the time when
${\cal A}$ is used, so that ${\cal V}_m\exp(H_m)e_1\|b\|$ is computed; the 
 middle column refers to the case when
$M_1$ is used, so that $x_m^\otimes$ in (\ref{eqn:xotimes}) is computed. As expected, the 
CPU time for $K_m(M_1,b_1)$ is several orders of magnitude smaller than for $K_m({\cal A},b)$.
In the latter case, timings became prohibitive for $n=3,000$, since the generation of the
basis for the space entails the orthogonalization of vectors in $\RR^{n^2}$. The computational
costs remain extremely low when computing a basis for $K_m(M_1,b_1)$.
The left plot of Figure \ref{fig:graphs} shows the convergence history of the two approaches,
in terms of space dimensions, when the smallest matrix in the set is used. Convergence is
monitored by measuring the difference between the last two iterates, as done in the previous
examples.  Once again, convergence is faster when the Kronecker form is exploited.
The right plot of Figure \ref{fig:graphs} reports the convergence history of  $x_m^\otimes$
for all matrices in the set. All spectra are roughly contained in the interval $[-15,15]$,
therefore the expected 
convergence rate is approximately the same for all matrices.
}
\end{example}

\begin{table}[htb]
\centering
\begin{tabular}{|ccr|}
\hline
   $n$     &  CPU Time      &    CPU Time \\
           & $K_m(M_1,b_1)$ &  $K_m({\cal A},b)$ \\ \hline
      1000 & 0.02662        &    29.996 \\
      2000 & 0.04480        &   189.991 \\
      3000 & 0.06545        &    -- \\
      4000 & 0.90677        &    -- \\
      5000 & 0.99206        &    -- \\
\hline
\end{tabular}
\caption{Example \ref{ex:graphs}. 
CPU Time for the construction of the Krylov approximation space of dimension $m=30$
when using either $M_1\in\RR^{n\times n}$ (symmetric) or ${\cal A} = M_1\otimes I + 
I\otimes M_1 \in \RR^{n^2\times n^2}$.
Only results for the smallest graph matrices $\cal A$ are reported when building
$K_m({\cal A},b)$.\label{tab:graphs}}
\end{table}

\begin{remark}
{\rm
Besides the exponential, the matrix sine and cosine are also well-behaved
with respect to the Kronecker sum structure. Indeed, the following
identities hold \cite[Theorem 12.2]{Higham2008}:
\begin{eqnarray}\label{eqn:sin_kron}
\sin(M_1\oplus M_2) = \sin(M_1)\otimes \cos(M_2) + \cos(M_1)\otimes \sin(M_2)
\end{eqnarray}
and
\begin{eqnarray}\label{eqn:cos_kron}
\cos(M_1\oplus M_2) = \cos(M_1)\otimes \cos(M_2) - \sin(M_1)\otimes \sin(M_2).
\end{eqnarray}
These identities can be exploited to greatly reduce the 
computational cost and storage requirements for the evaluation of
$f({\cal A})b$ when $f$ is either the sine or cosine or a combination of
these functions.
}
\end{remark}

\subsection{Convergence considerations}\label{sec:conv}
An expression for the error can be deduced by using the
form in (\ref{eqn:kron_form}). Indeed, letting
$x^{(1)} = \exp(M_1)b_1$ and $x^{(2)} = \exp(M_2)b_2$, and also
$X=x^{(1)} (x^{(2)})^T$ and 
$X_m^{\otimes}=x_m^{(1)} (x_m^{(2)})^T$,
it holds
\begin{eqnarray}
\|\exp({\cal A})b-x_m^{\otimes}\| &=&
\|X-X_m^{\otimes}\|_F \nonumber \\
&=&
\| x^{(1)} ( x^{(2)} - x_m^{(2)})^T + (x^{(1)} -x_m^{(1)}) (x_m^{(2)})^T \|_F \nonumber \\
&\le & 
\| x^{(1)}\| \, \| x^{(2)} - x_m^{(2)}\| + \|x^{(1)} -x_m^{(1)}\| \, \|x_m^{(2)}\| . \label{eqn:error_exp}
\end{eqnarray}
Therefore, the error norm in the approximation to $\exp({\cal A})b$ is bounded by
the errors of the two separate approximations with $M_1$
and $M_2$.
The relation (\ref{eqn:error_exp}) can also be used for deriving 
a priori convergence bounds for $\exp({\cal A}) b$ in terms of
the bounds for $M_1$ and $M_2$. We will give such bounds in the case
$M_1$ and $M_2$ are Hermitian and positive definite. The following result was proved in
\cite{Hochbruck1997}.

\begin{theorem} \label{th:HL}
Let $M$ be a Hermitian positive semidefinite matrix with eigenvalues
in the interval $[0,4\rho]$. Then the error in the Arnoldi approximation
of $\exp(\tau M) v$ with $\|v\|=1$, namely,
$\varepsilon_m:= \|\exp(-\tau M) v - V_m \exp(-\tau T_m) e_1 \|$,
 is bounded in the following ways:
\begin{enumerate}
\item[i)] $\varepsilon_m \le
10 \exp(-m^2/(5\rho\tau))$, for $\rho \tau \ge 1$ and $\sqrt{4\rho\tau}\le m 
\le 2\rho\tau$;
\item[ii)] $\varepsilon_m \le
10 (\rho\tau)^{-1} \exp(-\rho\tau) \left ( \frac{e\rho\tau}{m}\right)^m$ for
$m\ge 2\rho\tau$.
\end{enumerate}
\end{theorem}
\vskip 0.1in

We next show how Theorem \ref{th:HL} can be used to compare the difference
in convergence rates between $x_m$ and $x_m^{\otimes}$. To simplify the
presentation, we assume that $M=M_1=M_2$, 
and that $b_1=b_2$, with $\|b_1\|=1$. We refer the reader to 
\cite{Beckermann.Reichel.09,Knizhnerman.91}
for estimates similar to those of Theorem \ref{th:HL}.

It can be shown that if $\lambda_i$, $i=1, \ldots, n$
are the eigenvalues of $M$ (in decreasing order), then the $n^2$ eigenvalues
of $\cal A$ are given by $\lambda_i+\lambda_j$, $i,j\in\{1, \ldots, n\}$; see,
e.g., \cite[Theorem 4.4.5]{Horn.Johnson.91}.
Therefore in particular, the largest and smallest eigenvalues of $\cal A$ equal $2\lambda_1$
and $2\lambda_n$, respectively.
If we apply Theorem \ref{th:HL} to $M-\lambda_n I$ for $\tau =1$, then we obtain that for $m$ large enough the
error is bounded as
$$
\varepsilon_m(M) \le 
10 \frac{\exp(-\rho)}{\rho} \left ( \frac{e\rho}{m}\right)^m, \qquad
\rho = \frac{\lambda_1-\lambda_n}{4} .
$$
On the other hand, Theorem \ref{th:HL} applied to ${\cal A}-2\lambda_n I$ yields 
$$
\varepsilon_m({\cal A}) \le
10 \frac{\exp(-\widehat\rho)}{\widehat\rho} \left ( \frac{e\widehat\rho}{m}\right)^m, \qquad
\widehat \rho = \frac{2\lambda_1-2\lambda_n}{4} = 2\rho .
$$
The ratio between the two bounds is given by
$$
\frac{\exp\left(\frac{\rho}{2}\right)}{2^{m-1}},
$$ 
which is
in favor of the computation with $M$ for small $\rho$. In case $\rho$ is very large,
both methods become very slow.

\section{The case of the matrix inverse} \label{sec:inv}
Also very important in applications is the case of the inverse,
$f({\cal A}) = {\cal A}^{-1}$ where $\cal A$ has Kronecker sum structure.
The solution of linear systems of the form ${\cal A}x = b$ arises in many
applications (PDEs, imaging, Markov chains, networks, etc.) and Krylov
subspace methods are widely used, typically together with preconditioning,
to solve such systems. The Kronecker structure can be readily exploited
if $b$ is the result of the vectorization of a low rank matrix. For
simplicity, let us assume that $b={\rm vec}(b_1b_2^T)$. Then the system ${\cal A}x = b$ 
is equivalent to the following Sylvester equation 
(see, e.g., \cite[Sec.~4.4]{Horn.Johnson.91}):
\begin{eqnarray}\label{eqn:Sylv}
M_1 X + X M_2^T = b_1 b_2^T, \qquad x = {\rm vec}(X) .
\end{eqnarray}
Numerical methods that exploit the small size of $M_1, M_2$ can be 
used to solve the linear matrix equation in (\ref{eqn:Sylv}). If $n_1$ and $n_2$ 
are of order up to a few
thousands, then the Bartels--Stewart algorithm can be used \cite{Bartels.Stewart.72}.
Otherwise, $X$ can be approximated by using different approaches,
depending on the relative size of $n_1$ and $n_2$; we refer the reader
to \cite{Simoncini.survey13} for detailed discussion of the available methods
and the related references. Here we briefly describe the idea of approximate solution
by projection onto an appropriate subspace, which will be used in section \ref{sec:CS}.
For simplicity of exposition we assume $M_2=M_1$ and $b_2=b_1$; if this is not the
case, straightforward modifications can be included; see \cite{Simoncini.survey13}.
If the orthonormal columns of $P_m$($=Q_m$) are a basis for the considered subspace of $\RR^n$ of dimension $m$, then
an approximation to $X$ is sought as $X \approx X_m = P_m Y_m P_m^T$, where $Y_m \in\RR^{m\times m}$
is determined by imposing additional conditions. 
A common strategy consists of imposing that the residual
$R_m : = M_1 X_m + X_m M_1^T - b_1 b_1^T$ 
be orthogonal to the generated subspace, that is,
$(P_m\otimes P_m)^T {\rm vec}(R_m) = 0$, or, in matrix terms, $P_m^TR_mP_m = 0$, where
a zero matrix appears on the right-hand side.
Substituting in this last equation the definition of $R_m$ and $X_m$, gives
$$
P_m^T M_1 P_m Y_m P_m^TP_m + P_m^T P_m Y_mP_m^T  M_1^T P_m - P_m^T b_1 b_1^TP_m = 0.
$$ 
Recalling that $P_m^TP_m=I$ and that $P_m^TM_1P_m = T_1$, we obtain the small scale
linear equation
$$
 T_1  Y_m +  Y_m T_1^T - \widehat  b_1 \widehat b_1 = 0, \qquad \widehat b_1 = P_m^T b_1,
$$ 
whose solution yields $Y_m$.
In particular,
\begin{eqnarray}\label{eqn:lyap_err}
\|x-x_m^{\otimes}\| =
\|X - P_m Y_m P_m^T \|_F,
\end{eqnarray}
where $\| \, \cdot \,\|_F$ denotes the Frobenius norm.
 While we refer to \cite{Simoncini.survey13} for a
detailed analysis, here we notice that the approximate solution $x_m^{\otimes}={\rm vec}(X_m)$
to $f({\cal A})b = {\cal A}^{-1} b$ can be written in the more familiar form
\begin{eqnarray}\label{eqn:lyap}
x_m^{\otimes} &= &
{\rm vec}(P_m Y_m P_m^T)  = (P_m \otimes P_m) {\rm vec}(Y_m) 
\nonumber\\
& =&
(P_m \otimes P_m) (T_1\otimes I + I \otimes T_1)^{-1} (P_m\otimes P_m)^T b.
\end{eqnarray}
This form will be used in section~\ref{sec:CS} to express the approximation
error  of Cauchy-Stieltjes functions.

\section{Completely monotonic functions} \label{sec:scm}

Both the matrix exponential (in the form $f({\cal A}) = \exp(-{\cal A})$) and
the inverse are special cases of an
important class of analytic functions, namely, the completely
monotonic functions \cite{Widder.46}. We recall the following definitions.

\vskip 0.01in

\begin{definition}
Let $f$ be defined in the interval $(a,b)$ where $-\infty \le a < b \le +\infty$.
Then, $f$ is said to be {\em completely monotonic} in $(a,b)$ if
$$(-1)^{k}f^{(k)} (x) \ge 0 \quad {\rm for\ all} \quad a < x < b \quad {\rm and\ all}
\quad k=0,1,2,\ldots $$
Moreover, $f$ is said to be {\em strictly completely monotonic}
in $(a,b)$ if
$$(-1)^{k}f^{(k)} (x) > 0 \quad {\rm for\ all} \quad a < x < b \quad {\rm and\ all}
\quad k=0,1,2,\ldots $$
\end{definition}
Here $f^{(k)}$ denotes the $k$th derivative of $f$, with $f^{(0)}\equiv f$.

An important theorem of Bernstein states that 
a function $f$ is completely monotonic
in $(0,\infty)$ if and only if $f$ is the Laplace--Stieltjes transform
of $\alpha (\tau)$;
\begin{equation}\label{bern}
f(x) = \int_0^\infty  e^{-\tau x}\, {\rm d}\alpha(\tau),
\end{equation}
where $\alpha (\tau)$ is nondecreasing and
the integral in (\ref{bern}) converges for all $x>0$.  
See \cite[Chapter 4]{Widder.46}.
For this reason, completely monotonic functions on $(0,\infty)$ are also
referred to as {\em Laplace--Stieltjes functions}.

Important examples of Laplace--Stieltjes functions include:

\begin{enumerate}
\item $f_1(x) = 1/x = \int_0^\infty e^{-x\tau} d\alpha_1(\tau)$ for $x>0$,
where $\alpha_1(\tau) = \tau$ for $\tau\ge 0$.
\item $f_2(x) = e^{-x} = \int_0^\infty e^{-x\tau} d\alpha_2(\tau)$ for $x>0$,
where $\alpha_2(\tau) = 0$ for $0\le \tau < 1$ and $\alpha_2(\tau) = 1$ for $\tau\ge 1$.
\item $f_3(x) = (1 - e^{-x})/x = \int_0^\infty {\rm e}^{-x\tau} d\alpha_3(\tau)$
for $x>0$,
where $\alpha_3 (\tau) = \tau$ for $0\le \tau \le 1$, and $\alpha_3(\tau) = 1$ for $\tau\ge 1$.
\end{enumerate}

Also, the functions
$x^{-\sigma}$ (for any $\sigma > 0$), $\log(1+1/x)$ and $\exp(1/x)$,
are all strictly completely monotonic on $(0,\infty)$.
Moreover, products and positive linear combinations
of strictly completely monotonic functions are strictly completely monotonic.

Formula (\ref{bern}) suggests the use of quadrature rules to approximate $f({\cal A})b$ when
$\cal A$ is a Kronecker sum 
and $f$ is strictly
completely monotonic on $(0,\infty)$:

\begin{equation}\label{eqn:quad}
f({\cal A})b = \int_0^{\infty} \exp(- \tau {\cal A})b\, {\rm d}{\alpha} (\tau)
\approx \sum_{k=1}^q w_k \exp(-\tau_k {\cal A})b,
\end{equation}
where $\tau_1,\ldots ,\tau_k\in (0,\infty)$ are suitably chosen quadrature 
nodes and $w_1,\ldots ,w_k\in \RR$ are the quadrature weights; see,
for example, \cite[Sec.~5]{Hackbusch}. As shown
in the previous section, the
Kronecker sum structure can be exploited in the computation of the 
individual terms $\exp(-\tau_k {\cal A})b$. Also note that each contribution
to the quadrature can be computed independently of the others, which could
be useful in a parallel setting. This approach could be especially useful 
in cases where convergence of the Krylov subspace approximation is slow. 

In the case of interpolatory quadrature and Hermitian matrices,  
the absolute error $\|f({\cal A})b - \sum_{k=1}^q w_k \exp(-\tau_k {\cal A})b\|$
(in the 2-norm) is easily seen to be bounded by $\varepsilon \|b\|$, where
$\varepsilon$ is 
defined as
$$\varepsilon = \max_{i,j} \left | \int_0^\infty \exp(- \tau (\lambda_i + \mu_j)) \,
{\rm d}{\alpha} (\tau)
- \sum_{k=1}^q w_k \exp(-\tau_k (\lambda_i + \mu_j))\right |
$$
and $\lambda_i, \mu_j$ range over the spectra of $M_1$, $M_2$, where
${\cal A} = M_2\otimes I + I\otimes M_1$.
In the case $M_1 = M_2=M=M^*$, the error can be bounded by
$$\varepsilon \le \max_{\lambda, \mu \in [\lambda_{\min}, \lambda_{\max}]}
\left | \int_0^\infty \exp(- \tau (\lambda + \mu)) \,
{\rm d}{\alpha} (\tau)
-\sum_{k=1}^q w_k \exp(-\tau_k (\lambda + \mu)) \right |,
$$
where $\lambda_{\min}, \lambda_{\max}$ are the extreme eigenvalues of $M$.
For additional discussion of error bounds associated with the use of quadrature rules
of the form (\ref{eqn:quad}), see \cite[Sec.~5.7]{Hackbusch}.

Analogous considerations apply to more general types of functions. 
For a function $f$ analytic inside a contour $\Gamma \in \CC$ containing
the eigenvalues of $\cal A$ in its interior and continuous on $\Gamma$
we can write
$$f({\cal A}) = \frac{1}{2\pi i} \int_{\Gamma} f(z)({\cal A}-zI)^{-1} \, {\rm d}z.$$
Quadrature rules can be used to obtain
approximations of the form
$$f({\cal A})b = \frac{1}{2\pi i} \int_{\Gamma} f(z) ({\cal A}-z I)^{-1}b\, {\rm d}z
\approx \sum_{k=1}^q w_k ({\cal A} - z_k I)^{-1}b,$$
requiring the solution of the $q$ linear systems $({\cal A} - z_k I)x = b$,
possibly in parallel for $k=1,\ldots ,q$. We refer to \cite{ThreeNicks} for 
details on how to apply this technique efficiently. Again, the Kronecker sum structure of
$\cal A$, if present, can be exploited to greatly reduce the
computational cost and storage requirements. In particular,
if $b = {\rm vec}(b_1b_2^T)$, then according to section \ref{sec:inv}, each 
system $({\cal A} - z_k I)x = b$ is equivalent to solving the
linear matrix equation $(M_1 - z_kI) X + X M_2^T = b_1 b_2^T$, with $x={\rm vec}(X)$.

{
Another important class of functions is given by the Cauchy--Stieltjes (or
Markov-type) functions, which can be written as
$$
f(z) = \int_\Gamma \frac {{\rm d}\gamma(\omega)} {z-\omega},
\quad z\in \CC \setminus \Gamma\,,
$$
where $\gamma$ is a (complex) measure supported on a closed set $\Gamma \subset \CC$
and the integral is absolutely convergent.
This class is closely related to, but distinct from, the class of
Laplace--Stieltjes functions; see \cite[Chapter VIII]{Widder.46} for a general
treatment.
In this paper we are especially interested in the particular case $\Gamma = (-\infty, 0]$
so that
\begin{eqnarray}\label{eqn:markov}
f(x) = \int_{-\infty}^0 \frac {{\rm d}\gamma(\omega)} {x - \omega}, 
\quad x\in \CC \setminus (-\infty, 0]\,,
\end{eqnarray}
where $\gamma$ is now a (possibly signed) real measure.
Important examples of Cauchy--Stieltjes function that are frequently 
encountered in applications (see \cite{Guettel.survey.13}) include
\begin{eqnarray*}
&& z^{-\frac 1 2} = \int_{-\infty}^0 \frac 1 {z-\omega} \frac 1 
{\pi \sqrt{-\omega}} {\rm d}\omega,
\\
&& \frac{e^{-t\sqrt{z}}-1}{z} = \int_{-\infty}^0 \frac 1 {z-\omega} \frac 
{\sin(t\sqrt{-\omega})}{-\pi \omega} {\rm d}\omega,
\\
&& \frac{\log(1+z)}{z} = \int_{-\infty}^{-1} \frac 1 {z-\omega} \frac 1 
{(-\omega)} {\rm d}\omega.
\end{eqnarray*}
}

\subsection{Convergence analysis for Laplace--Stieltjes functions}\label{sec:LS}
For Laplace--Stieltjes functions and symmetric positive definite matrices,
in this section we analyze the convergence rate of the approximation obtained
by exploiting the Kronecker form. Moreover, we compare this rate with that 
of the standard approximation with $\cal A$, and that of the approximation of
$M_1$. We will mainly deal with standard Krylov approximations, as 
error estimates for the exponential are available. A few comments are also 
included for rational Krylov subspaces.

\subsubsection{Analysis of Krylov subspace approximation}
In this section we show that the convergence rate of the 
approximation when using $x_m^{\otimes}$ is smaller than
that with $x_m = {\cal V}_mf(H_m){\cal V}_m^Tb$. Moreover,
it is also smaller than the convergence rate of the approximation
$x_m^{(1)}$ to $f(M_1)b_1$.
%
For simplicity of exposition we assume that $M_1=M_2$ and $b_1=b_2$, with
$\|b_1\|=1$.

\begin{proposition}\label{prop:LS}
Let $f$ be a Laplace--Stieltjes function, and
$x_m^{\otimes}=(P_m\otimes P_m) f({\cal T}_m) (P_m\otimes P_m)^T  b$ be
the Kronecker approximation to $f({\cal A})b$. 
Moreover, let $x^{(1)}=\exp(-\tau M_1) b_1$ and $x_m^{(1)} = P_m\exp(-\tau T_1)P_m^T b_1$.
We also define the scaled quantity
$\widehat x^{(1)} := e^{-\lambda_{\min} \tau} x^{(1)} =
e^{- (M_1+\lambda_{\min}I)\tau} b_1$; analogously for $\widehat x_m^{(1)}$.
Then
\begin{eqnarray*}
\|f({\cal A})b - x_m^{\otimes}\|  \le 
2 \int_{0}^\infty 
\|\widehat x^{(1)} - \widehat x_m^{(1)} \| \,{\rm d}\alpha(\tau)  .
\end{eqnarray*}
\end{proposition}

\begin{proof}
Recalling the notation of (\ref{eqn:xotimes}) and  (\ref{eqn:kron_form}) leading 
to (\ref{eqn:error_exp}), we have
\begin{eqnarray*}
\|f({\cal A})b - x_m^{\otimes}\| & = &
\| \int_0^\infty 
(e^{-\tau {\cal A}}b -  (P_m\otimes P_m) e^{-\tau {\cal T}_m} (P_m\otimes P_m)^Tb)
{\rm d}\alpha(\tau) \| \\
&\le&
 \int_0^\infty (\|x^{(1)}\|+\|x_m^{(1)}\|) \, \|x^{(1)} -x_m^{(1)}\| 
{\rm d}\alpha(\tau) \\
&\le &
 \int_0^\infty 2 e^{-\lambda_{\min}\tau}  \|x^{(1)} -x_m^{(1)}\| 
{\rm d}\alpha(\tau) \\
&=&
2 \int_{0}^\infty 
\|e^{-\tau(M_1+\lambda_{\min}I)}b_1 -P_me^{-\tau(T_1+\lambda_{\min}I)} P_m^T b_1\| \,{\rm d}\alpha(\tau) \\
&=&
2 \int_{0}^\infty 
\|\widehat x^{(1)} - \widehat x_m^{(1)} \| \,{\rm d}\alpha(\tau) ,
\end{eqnarray*}
where in the last inequality we have used $\lambda_{\min}(M_1) \le \lambda_{\min}(T_1)$,
so that 
$$
\|x_m^{(1)}\| \le 
e^{-\lambda_{\min}(T_1) \tau} \le e^{-\lambda_{\min}(M_1) \tau}.
$$
\end{proof}

We remark that the extra shift in the matrix $M_1$ is what makes the solution $x_m^{\otimes}$
converge faster than $x_m^{(1)}$.

In light of Proposition \ref{prop:LS},
bounds for the error norm can be found by estimating the
error norm in the approximation of the exponential function under the measure d$\alpha$.
Depending on the function $\alpha(\tau)$, different approximation
strategies need to be devised. Here we analyze the case
d$\alpha(\tau) = \frac 1 {\tau^\gamma} d\tau$ for some $\gamma \ge 0$; for instance,
the function $f(x) = 1/\sqrt{x}$ falls in this setting with $\gamma = 3/2$. Then
$$
\|f({\cal A})v - x_m^{\otimes}\|  \le 
 2\int_0^\infty  \frac{1}{\tau^\gamma}  \|\widehat x^{(1)} -\widehat x_m^{(1)}\| 
{\rm d}\tau .
$$

The case $\gamma=0$ is special. Since $\tau^\gamma=1$, the integral on the right-hand side
can be bounded as in \cite[proof of (3.1)]{Simoncini.Druskin.09}, so that it holds
\begin{eqnarray*}
\|f({\cal A})v - x_m^{\otimes}\| & \le &
 2\frac{\sqrt{\widehat\kappa}+1}{\lambda_{\min}\sqrt{\widehat\kappa}}
 \left( \frac{\sqrt{\widehat \kappa}-1}{\sqrt{\widehat \kappa}+1}\right )^m ,
\end{eqnarray*}
where $\widehat \kappa = (\lambda_{\max}+\lambda_{\min})/(\lambda_{\min}+\lambda_{\min})$.

We focus next on the case $\gamma> 0$. We split the integral as
\begin{eqnarray}
 \int_0^\infty  \frac{1}{\tau^\gamma}  \|\widehat x^{(1)} -\widehat x_m^{(1)}\| 
{\rm d}\tau &=&
 \int_0^{\frac{m^2}{4\rho}} \frac{1}{\tau^\gamma}  \|\widehat x^{(1)} -\widehat x_m^{(1)}\|\,  {\rm d}\tau
+ \int_{\frac{m^2}{4\rho}}^\infty  \frac{1}{\tau^\gamma}  \|\widehat x^{(1)} -\widehat x_m^{(1)}\| \, {\rm d}\tau
\nonumber\\
&=:& I_1 + I_2 .\label{eqn:I1I2}
\end{eqnarray}

\begin{lemma}\label{lemma:gamma=0}
With the previous notation, for $\gamma>0$ it holds
$$
I_2 \le \left(\frac{4\rho}{m^2}\right)^\gamma \frac{\sqrt{\widehat\kappa}+1}{\lambda_{\min}\sqrt{\widehat\kappa}}
 \left( \frac{\sqrt{\widehat \kappa}-1}{\sqrt{\widehat \kappa}+1}\right )^m ,
$$
where $\widehat \kappa = (\lambda_{\max}+\lambda_{\min})/(\lambda_{\min}+\lambda_{\min})$.
\end{lemma}

\begin{proof}
We note that in the given interval $\tau^\gamma\ge \left(\frac{m^2}{4\rho}\right)^\gamma$, so that
\begin{eqnarray*}
I_2 
\le 
 \left(\frac{4\rho}{m^2}\right)^\gamma
\int_{\frac{m^2}{4\rho}}^\infty   \|\widehat x^{(1)} -\widehat x_m^{(1)}\| \,{\rm d}\tau \le
 \left(\frac{4\rho}{m^2}\right)^\gamma
\int_{0}^\infty   \|\widehat x^{(1)} -\widehat x_m^{(1)}\| \,{\rm d}\tau.
\end{eqnarray*}
Following \cite[Prop.3.1]{Simoncini.Druskin.09}, an explicit bound for the last integral can be
obtained, from which the bound follows.
\end{proof}

The derivation of an upper bound for $I_1$ in (\ref{eqn:I1I2}) is a little more involved. 
\vspace{0.1in}

\begin{lemma}
With the previous notation, for $\gamma>0$ it holds
{\footnotesize
\begin{eqnarray*}
I_1 \le 10 \left(
\frac 1 \rho \left( \frac{e\rho}{m}\right)^m 
\left( \frac{1}{2\lambda_{\min}+\rho}\right)^{m-\gamma} 
{\pmb\gamma}\left(m-\gamma, \frac{2\lambda_{\min}+\rho}{2\rho} m\right)  \right .
%
 +\left .
\left(\frac{2\rho}{m}\right)^{\gamma-\frac 1 2}\!\! \left(\frac{\pi}{2\lambda_{\min}}\right)^{\frac 1 2}
e^{-2m\sqrt{\frac{2\lambda_{\min}}{5\rho}}} 
\right ) ,
\end{eqnarray*} 
}
where $\pmb\gamma$ is the lower incomplete Gamma function.
\end{lemma}

\begin{proof}
We observe that the quantity $\|\widehat x^{(1)} -\widehat x_m^{(1)}\|$  in $I_1$
can be bounded by using Theorem \ref{th:HL} (see \cite{Hochbruck1997}), hence
we further split $I_1$ as
\begin{eqnarray}\label{eqn:I1split}
I_1 = 
 \int_0^{\frac{m}{2\rho}} \frac{1}{\tau^\gamma}  \|\widehat x^{(1)} -\widehat x_m^{(1)}\| \,{\rm d}\tau +
 \int_{\frac{m}{2\rho}}^{\frac{m^2}{4\rho}} \frac{1}{\tau^\gamma}  \|\widehat x^{(1)} -\widehat x_m^{(1)}\| \,{\rm d}\tau .
\end{eqnarray}
Theorem \ref{th:HL} can be applied to positive semidefinite matrices. Therefore we 
write $e^{-(M_1+\lambda_{\min}I)\tau} = e^{-2\lambda_{\min}\tau} e^{-(M_1-\lambda_{\min}I)\tau}$, with
$M_1-\lambda_{\min}I$ positive semidefinite.
For the first integral in (\ref{eqn:I1split}) we thus have (see Theorem \ref{th:HL}(ii))
\begin{eqnarray*}
 \int_0^{\frac{m}{2\rho}} \frac{1}{\tau^\gamma}  \|\widehat x^{(1)} -\widehat x_m^{(1)}\| \,{\rm d}\tau 
&=&
 \int_0^{\frac{m}{2\rho}} \frac{e^{-2\lambda_{\min}\tau}}{\tau^\gamma}  \|e^{-(M_1-\lambda_{\min}I)\tau}b - 
(P_m\otimes P_m)e^{-({\cal T}_m-\lambda_{\min}I)\tau} \widehat b\| \,{\rm d}\tau  \\
&\le&
 10\int_0^{\frac{m}{2\rho}} \frac{e^{-(2\lambda_{\min}+\rho)\tau}}{\rho\tau^{\gamma+1}}  \left ( \frac{e\rho \tau}{m}\right)^m \,{\rm d}\tau  \\
&=&
 10\frac 1 \rho \left ( \frac{e\rho}{m}\right)^m \int_0^{\frac{m}{2\rho}} 
\tau^{m-\gamma-1} e^{-(2\lambda_{\min}+\rho)\tau}\,{\rm d}\tau \\
&= &
10\frac 1 \rho \left( \frac{e\rho}{m}\right)^m 
\left( \frac{1}{2\lambda_{\min}+\rho}\right)^{m-\gamma} 
{\pmb\gamma}\left(m-\gamma, \frac{2\lambda_{\min}+\rho}{2\rho} m\right)  .
\end{eqnarray*}

For the second integral in (\ref{eqn:I1split}), after the same spectral transformation and
also  
using $\frac 1{\tau^\gamma}\le \frac{1}{(m/(2\rho))^\gamma}$ for 
$\tau \in [\frac{m}{2\rho}, \frac{m^2}{4\rho}]$, we obtain
\begin{eqnarray*}
 \int_{\frac{m}{2\rho}}^{\frac{m^2}{4\rho}} \frac{1}{\tau^\gamma}  \|\widehat x^{(1)} -\widehat x_m^{(1)}\| \,{\rm d}\tau &\le &
 10\int_{\frac{m}{2\rho}}^{\frac{m^2}{4\rho}} \frac{e^{-2\lambda_{\min}\tau}}{\tau^\gamma}  e^{-\frac{m^2}{5\rho \tau}} \,{\rm d}\tau \\
&\le&
 10\left(\frac{2\rho}{m}\right)^{\gamma-1/2} \int_{\frac{m}{2\rho}}^{\frac{m^2}{4\rho}} \frac{1}{\tau^{\frac 1 2}}
e^{-2\lambda_{\min}\tau-\frac{m^2}{5\rho \tau}} \,{\rm d}\tau  \\
&\le&
 10\left(\frac{2\rho}{m}\right)^{\gamma-1/2} \int_{0}^{\infty} \frac{1}{\tau^{\frac 1 2}}
e^{-2\lambda_{\min}\tau-\frac{m^2}{5\rho \tau}} {\rm d}\tau.
\end{eqnarray*}
We then use \cite[formula 3.471.15]{Gradshteyn.Ryzhik.2007}, 
namely $\int_0^\infty x^{-\frac 1 2} e^{-\beta_1 x - \frac{\beta_2}{x}} dx =
\sqrt{\frac{\pi}{\beta_1}} e^{-2\sqrt{\beta_1\beta_2}}$, to finally write
$$
 10\left(\frac{2\rho}{m}\right)^{\gamma-1/2} \int_{0}^{\infty} \frac{1}{\tau^{\frac 1 2}}
e^{-2\lambda_{\min}\tau-\frac{m^2}{5\rho \tau}} {\rm d}\tau =
10\left(\frac{2\rho}{m}\right)^{\gamma-1/2} \sqrt{\frac{\pi}{2\lambda_{\min}}} e^{-2\sqrt{\frac{2\lambda_{\min} m^2}{5\rho}}} .
$$
\end{proof}

By collecting all bounds we can prove a final upper bound for the error,
and give its asymptotic convergence rate. To this end, we first need the following
technical lemma, whose proof is given in the appendix.

\begin{lemma}\label{lemma:gamma}
For $0<x \le {\alpha} n$ with $ 0 <\alpha < 1$ it holds that
$${\pmb \gamma}(n,x)  \le \frac{1}{1-\alpha}  \frac{x^n}{n} e^{-x}.$$
\end{lemma}

\begin{theorem}
For $\gamma> 0$ and with the notation above, it holds that
\begin{eqnarray}
 \|f({\cal A})v  &-& x_m^{\otimes}\|  \le 2(I_1 + I_2) \nonumber\\
&& \le 20\left(
\frac 1 \rho \left( \frac{e\rho}{m}\right)^m 
\left( \frac{1}{2\lambda_{\min}+\rho}\right)^{m-\gamma} 
{\pmb\gamma}\left(m-\gamma, \frac{2\lambda_{\min}+\rho}{2\rho} m\right) \right .
\nonumber \\
&& +\left .
\left(\frac{2\rho}{m}\right)^{\gamma-1/2} \sqrt{\frac{\pi}{2\lambda_{\min}}}\, \, e^{-2m\sqrt{\frac{2\lambda_{\min}}{5\rho}}} 
\right )\nonumber  \\
&& +
20\left(\frac{4\rho}{m^2}\right)^\gamma \frac{\sqrt{\widehat\kappa}+1}{\lambda_{\min}\sqrt{\widehat\kappa}}
 \left( \frac{\sqrt{\widehat \kappa}-1}{\sqrt{\widehat \kappa}+1}\right )^m \nonumber \\
&& = {\cal O}\left(\exp\left(-\frac{2m}{\sqrt{\widehat\kappa}}\right )\right) \qquad \mbox{\rm for\, $m$ and $\widehat\kappa$ \,large \,enough}. \label{eqn:Boundgamma_ge_1}
\end{eqnarray}
\end{theorem}

\begin{proof}
We only need to show that the term involving $\pmb\gamma$ is asymptotically bounded above by
${\cal O}\left(\exp\left(-\frac{2m}{\sqrt{\widehat\kappa}}\right )\right)$
for $m$ and $\widehat\kappa$ large.
Simple calculations show that the second argument of $\pmb\gamma$
satisfies $(2\lambda_{\min}+\rho)/(2\rho) m \le \alpha m$ with $\alpha <1$ for
$\lambda_{\max}/\lambda_{\min} > 9$; the larger this eigenvalue ratio, the smaller $\alpha$,
so that for a large ratio, the bound $(2\lambda_{\min}+\rho)/(2\rho) m \le \alpha (m-\gamma)$ 
also holds.
Hence, we can use Lemma \ref{lemma:gamma} to write
${\pmb\gamma}(n,x) \approx e^{-x} \frac{x^n}{n}$.
Therefore\footnote{We assume here that $m-\gamma$ is a positive integer, otherwise we can
take the Gamma function associated with
the closest integer larger than $m-\gamma$.},
\begin{eqnarray*}
\frac {10} \rho &&\left( \frac{e\rho}{m}\right)^m 
\left( \frac{1}{2\lambda_{\min}+\rho}\right)^{m-\gamma} 
{\pmb\gamma}\left(m-\gamma, \frac{2\lambda_{\min}+\rho}{2\rho} m\right) \\
&\approx &
\frac {10}{\rho} \left( \frac{e\rho}{m}\right)^m 
\left( \frac{1}{2\lambda_{\min}+\rho}\right)^{m-\gamma} 
\left(\frac{2\lambda_{\min}+\rho}{2\rho} m\right)^{m-\gamma}
\frac 1 {m-\gamma} e^{-\frac{2\lambda_{\min}+\rho}{2\rho} m}\\
&= &
10 
\frac{2^\gamma\rho^{\gamma-1}}{m^\gamma(m-\gamma)} 
\left(\frac{e}{2}\right)^{m}
e^{-\frac{2\lambda_{\min}+\rho}{2\rho} m}
= 10 
\frac{2^\gamma\rho^{\gamma-1}}{m^\gamma(m-\gamma)} 
e^{-\frac{2\lambda_{\min}+(\ln 4 -1)\rho}{2\rho}m}.
\end{eqnarray*}
Writing down $\rho =\frac 1 4 (\lambda_{\max}-\lambda_{\min})$, and after
a few algebraic calculations we obtain
{\footnotesize
$$
\frac{2\lambda_{\min} + (\ln 4 -1)\rho}{2\rho} 
= \frac 1 2 \frac{ (\ln 4 -1)\lambda_{\max} + (8 - (\ln 4 -1)) \lambda_{\min}}{\lambda_{\max}-\lambda_{\min}}
= \frac 1 2 \left ( \frac{4}{\widehat \kappa - 1} + (\ln 4 -1)\right) \ge 
\frac{2}{\sqrt{\widehat \kappa}}  ,
$$ }
where the last inequality holds for $\widehat \kappa$ large enough (namely for $\widehat \kappa \ge 10$).
\end{proof}

The theorem above states that the convergence rate of the approximation
depends on the condition number of the shifted matrix.

\begin{remark}
{\rm
If $f({\cal A})b$ is approximated in the Krylov subspace $K_m({\cal A},b)$, then the error norm can
be written as
$$
\|f({\cal A})b - x_m\| 
\le \int_0^\infty \|e^{-\tau {\cal A}}b - {\cal V}_m e^{-\tau H_m}{\cal V}_m^T b\| {\rm d}\alpha.
$$
Therefore, all the previous steps can be replicated, leading to an estimate of the type
$$
\|f({\cal A})b - x_m\| 
 \approx \exp\left(-\frac{2m}{\sqrt{\kappa}}\right ) 
$$
where now $\kappa = \lambda_{\max}/\lambda_{\min}$. The improvement in the convergence rate when
exploiting the Kronecker form thus becomes readily apparent, with the shift acting as an
``accelerator''. It is also important to realize that the error norm
$\|f(M_1)b_1 - x_m^{(1)}\|$ is also driven by the same quantity $\kappa$, since the condition number
of $M_1$ and $\cal A$ is the same. Therefore, it is really by using the Kronecker form that
convergence becomes faster.
$\square$
}
\end{remark}

\vskip 0.1in
We next illustrate our findings with a simple example. A diagonal matrix is considered so
as to be able to compute exact quantities, while capturing the linear convergence of the
approximation.

{\sc Example}. We consider $f(x)=1/\sqrt{x}$ (so that $\gamma=3/2$)
 and a diagonal matrix $M_1$ of size $n=500$ with logarithmically
distributed eigenvalues in $[10^1,10^3]$, giving $\rho \approx 247$; $b_1$ is the vector of all ones,
normalized to have unit norm. We wish to approximate
$f({\cal A})b$, with $b={\rm vec}(b_1 b_1^T)$. We compare the convergence curves
of the usual approximation $x_m = {\cal V}_m f(H_m)e_1$ (solid thin line), with 
that of $x_m^{\otimes} = (Q_m\otimes Q_m) f({\cal T}_m) \widehat b$ (dashed thick line). As expected, the
convergence rate of the latter is better (smaller) than for the standard method. 
The estimate in (\ref{eqn:Boundgamma_ge_1})
is reported as thick crosses, and it well approximates the correct convergence slope.
For completeness we also report the error 
norm $\|f(M_1)b_1 - Q_mf(T_1)Q_m^Tb_1\|$, which behaves like that
of the standard method, as noticed in the previous remark. 
$\square$

\begin{figure}[thb]
\centering
\includegraphics[width=2.5in,height=2.5in]{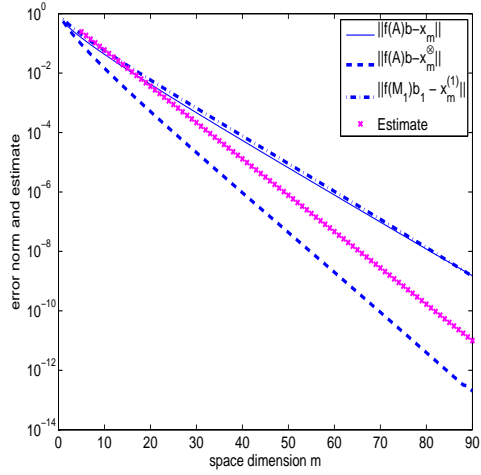}
\caption{Convergence curves for various approximation  methods, and estimate in (\ref{eqn:Boundgamma_ge_1}).
\label{fig:kronLS_f3}}
\end{figure}

\subsubsection{Rational Krylov subspace approximation}
Convergence bounds of Laplace-Stieltjes functions are harder to obtain
when rational Krylov subspaces are used, for few results on error bounds
for the exponential functions are available. In this section we discuss
some of the results that can be obtained. However, we will show that
additional results
can be derived for the subclass of Cauchy-Stieltjes functions, since they
involve inverses in place of exponential functions. 

If the approximation space is the extended Krylov subspace, which can be defined as
$K_m(M_1, b_1)+K_m(M_1^{-1},M_1^{-1}b_1)$, error bounds are difficult to obtain
for Laplace--Stieltjes functions, unless d$\alpha = $d$\tau$ (that is, $\gamma=0$). Indeed,
for $\gamma=0$, we have
{\footnotesize
$$
f({\cal A})b - x_m^{\otimes} = {\rm vec}\left ( \int_0^\infty \left(\exp(-\tau M_1) b_1 b_1^T \exp(-\tau M_1) - 
P_m \exp(-\tau{\cal T}_m)\widehat b_1 \widehat b_1^T  \exp(-\tau{\cal T}_m)P_m^T\right) {\rm d} \tau\right ) .
$$
}
The integral coincides with the error matrix in the approximation of the numerical solution
to the Lyapunov equation in the space Range($P_m$). 
This connection was already used in Lemma \ref{lemma:gamma=0} to establish an upper bound for
the Krylov subspace approximation. Here, we just mention that an asymptotic bound can be
obtained by using results in \cite{Knizhnerman.Simoncini.11,Kressner.Tobler.10,Beckermann.11},
giving
$$
\|f({\cal A})b - x_m^{\otimes}\| \approx 
{\cal O}\left ( \left(\frac{ \sqrt[4]{\kappa} -1}{\sqrt[4]{\kappa}+1}\right)^m\right),
$$
where $\kappa = \lambda_{\max}({\cal A})/\lambda_{\min}({\cal A}) =
 \lambda_{\max}(M_1)/\lambda_{\min}(M_1)$.  Note that while there is no beneficial 
shifted matrix in this bound, the fourth root of $\kappa$ appears, ensuring significantly faster
convergence rate than for standard Krylov subspaces.

For $\gamma >0$, lack of explicit error bounds for the exponential function leaves the
derivation of bounds for our setting an open problem. Nonetheless, experimental evidence
seems to suggest a convergence rate similar to the one for $\gamma=0$.

Upper bounds for the exponential function when rational Krylov subspaces (\ref{eqn:rks})
are employed are
usually asymptotic and specialized to optimal values of the 
parameters $\sigma_1, \ldots, \sigma_{m-1}$.
We refer the reader to \cite[sec.4.2]{Guettel.survey.13} for more details.

\subsection{Convergence analysis for Cauchy--Stieltjes functions}\label{sec:CS}
Functions belonging to
the Cauchy--Stieltjes class provide a more favorable setting, as they
are based on the resolvent function. Assume then that $f$ is a Cauchy-Stieltjes
function. Using the definition in (\ref{eqn:markov}) and assuming that 
$b={\rm vec}(b_1b_1^T)$, we can write
\begin{eqnarray*}
f({\cal A})b - x_m^{\otimes}  &=&
f({\cal A})b - (P_m\otimes P_m) f({\cal T}_m) (P_m\otimes P_m)^T b  \\
& = &\int_{-\infty}^0 ({\cal A}-\omega I)^{-1}b {\rm d}\gamma(\omega) -
(P_m\otimes P_m) 
\int_{-\infty}^0 
({\cal T}_m-\omega I)^{-1} (P_m\otimes P_m)^Tb {\rm d}\gamma(\omega) \\
&=&
\int_{-\infty}^0 \left(
 ({\cal A}-\omega I)^{-1}b  -
(P_m\otimes P_m) 
({\cal T}_m-\omega I)^{-1} (P_m\otimes P_m)^Tb \right)
{\rm d}\gamma(\omega)  .
\end{eqnarray*}
Let $M(\omega) = M_1 - \frac{\omega}{2} I$, so that for any $v$ and $V$ such
that $v={\rm vec}(V)$ we can write
$({\cal T}_m - \omega I) v = {\rm vec}( M(\omega) V + V M(\omega)^T)$. 
Then, recalling the derivation in (\ref{eqn:lyap}), we obtain
\begin{eqnarray*}
f({\cal A})b - x_m^{\otimes}  =
\int_{-\infty}^0  {\rm vec}( X(\omega) - X_m(\omega)) {\rm d}\gamma(\omega) ,
\end{eqnarray*}
where $X(\omega)$ and $X_m(\omega)$ are the exact and approximate
solutions to the linear matrix equation
$M(\omega) X + X M(\omega)^T = b_1 b_1^T$ when the space
$K_m(M_1-\frac{\omega}{2} I, b_1)$ is used. Note that this space is invariant
under shift, that is $K_m(M_1,b_1)=K_m(M_1-\omega I,b_1)$,
 so that the columns of $P_m$ are still a basis for this space.
An upper bound for the error can then be obtained as
\begin{eqnarray*}
\|f({\cal A})b - x_m^{\otimes}\| \le 
\int_{-\infty}^0  \| X(\omega) - X_m(\omega)\|_F {\rm d}\gamma(\omega) .
\end{eqnarray*}
Available convergence bounds for the approximation of $X(\omega)$ onto standard
and rational Krylov subspaces  can be employed as a first step towards
an upper bound for the error norm above. Final bounds will then be
obtained for specific choices of $f$, which yield specific functions $\gamma(\omega)$.

For standard Krylov subspaces, we can once again use
\cite[Proposition 3.1]{Simoncini.Druskin.09} to get
\begin{eqnarray*}
\|f({\cal A})b - x_m^{\otimes}\| \le 
\int_{-\infty}^0  
2\frac{ \sqrt{\widehat \kappa_\omega} +1}{(\lambda_{\min}-\frac 1 2\omega) \sqrt{\widehat\kappa_\omega}}
\left( \frac{\sqrt{\widehat \kappa_\omega}-1} {\sqrt{\widehat \kappa_\omega}+1}\right)^m
{\rm d}\gamma(\omega) ,
\end{eqnarray*}
where $\widehat\kappa_\omega=
(\lambda_{\max}+\lambda_{\min}-\frac 1 2 \omega)/ (\lambda_{\min}+\lambda_{\min}-\frac 1 2 \omega)$. 
For consistency with the previous notation, we shall use $\widehat\kappa_0 = \widehat\kappa$. Therefore,
\begin{eqnarray}\label{eqn:C.-S._bound}
\|f({\cal A})b - x_m^{\otimes}\| \le 
2
\left( \frac{\sqrt{\widehat \kappa}-1} {\sqrt{\widehat \kappa}+1}\right)^m
\int_{-\infty}^0  
\frac{ \sqrt{\widehat \kappa_\omega} +1}{(\lambda_{\min}-\frac 1 2\omega) \sqrt{\widehat\kappa_\omega}}
{\rm d}\gamma(\omega) ,
\end{eqnarray}
where 
$\left( \frac{\sqrt{\widehat \kappa}-1} {\sqrt{\widehat \kappa}+1}\right)^m \approx \exp(-2m/\sqrt{\widehat \kappa})$
for $\sqrt{\widehat \kappa}$ large.
The final upper bound can be obtained once the measure 
${\rm d}\gamma(\omega)$ is made more explicit, and this may influence the integration interval
as well. 
For instance, for $f(x)=x^{-\frac 1 2}$, 
\begin{eqnarray*}
\int_{-\infty}^0  
\frac{ \sqrt{\widehat \kappa_\omega} +1}{(\lambda_{\min}-\frac 1 2\omega) \sqrt{\widehat\kappa_\omega}}
{\rm d}\gamma(\omega) &=& 
\int_{-\infty}^0  
\frac{ \sqrt{\widehat \kappa_\omega} +1}{(\lambda_{\min}-\frac 1 2\omega) \sqrt{\widehat\kappa_\omega}}
\frac 1 {\sqrt{-\omega}} {\rm d}\omega  \\
& \le &
2\int_{-\infty}^0  
\frac{ 1}{(\lambda_{\min}-\frac 1 2\omega) }
\frac 1 {\sqrt{-\omega}} {\rm d}\omega  \\
& = &
2 \int_{-\infty}^{-1}  
\frac{ 1}{(\lambda_{\min}-\frac 1 2\omega) }
\frac 1 {\sqrt{-\omega}} {\rm d}\omega  +
2 \int_{-1}^{0}  
\frac{ 1}{(\lambda_{\min}-\frac 1 2\omega) }
\frac 1 {\sqrt{-\omega}} {\rm d}\omega  \\
& \le &
 \int_{-\infty}^{-1}  
\frac{ 1}{(-\omega) }
\frac 1 {\sqrt{-\omega}} {\rm d}\omega  +
2
\frac{ 1}{\lambda_{\min} }
 \int_{-1}^{0}  
\frac 1 {\sqrt{-\omega}} {\rm d}\omega   = 2 + \frac{4}{\lambda_{\min}}.
\end{eqnarray*}
Since the inverse square root is both a Laplace-Stieltjes function and a
Cauchy-Stieltjes function, it is not surprising that we get a similar
convergence rate. The setting of this section allows one to determine
a simpler expression for the bound.

Following similar steps as for the inverse square root, for $f(x) = \ln(1+x)/x$ we have (note the change
in the integration interval)
\begin{eqnarray*}
\int_{-\infty}^{-1}  
\frac{ \sqrt{\widehat \kappa_\omega} +1}{(\lambda_{\min}-\frac 1 2\omega) \sqrt{\widehat\kappa_\omega}}
{\rm d}\gamma(\omega) &=& 
\int_{-\infty}^{-1}  
\frac{ \sqrt{\widehat \kappa_\omega} +1}{(\lambda_{\min}-\frac 1 2\omega) \sqrt{\widehat\kappa_\omega}}
\frac 1 {-\omega} {\rm d}\omega  \\
& \le &
2 \int_{-\infty}^{-1}  
\frac{ 1}{(\lambda_{\min}-\frac 1 2\omega) }
\frac 1 {(-\omega)} {\rm d}\omega  \le 2\cdot 2 = 4.
\end{eqnarray*}
In both cases, the integral appearing in (\ref{eqn:C.-S._bound}) is bounded by a constant of modest size,
unless $\lambda_{\min}$ is tiny in the inverse square root.
Summarizing, when using standard Krylov subspace approximation,
$\|f({\cal A})b - x_m^{\otimes}\|$ is bounded by a quantity whose asymptotic term is
$\left( \frac{\sqrt{\widehat \kappa}-1} {\sqrt{\widehat \kappa}+1}\right)^m$ as $m$ grows.

When using the extended Krylov subspace for 
an Hermitian positive definite $M$, 
this quantity is replaced by
$\left( \frac{\sqrt[4]{\kappa}-1} {\sqrt[4]{\kappa}+1}\right)^m$, where
$m$ is the subspace dimension and $\kappa=\lambda_{\max}/\lambda_{\min}$
(see, e.g., \cite{KnizhnermanSimoncini2010},\cite{Kressner.Tobler.10}), 
 as conjectured in the case of Laplace-Stieltjes functions.
Here, an explicit upper bound can actually be obtained.

Rational Krylov subspaces can also be used, and the term $\|X-X_m^{\otimes}\|_F$ can
be estimated using, e.g., \cite[Theorem 4.9]{Druskin.Knizhnerman.Simoncini.11}.

\section{Conclusions} \label{sec:con}

In this paper we have shown how to take advantage of the Kronecker sum
structure in $\cal A$ when using Krylov subspace
methods to evaluate expressions of the form $f({\cal A})b$.
Special attention has been devoted to the important case of the
matrix exponential. Numerical experiments demonstrate that considerable
savings can be obtained when the Kronecker sum structure is exploited.
A detailed analysis of the convergence rate of the
new approximation for symmetric (or Hermitian) and 
positive definite matrices was also proposed.

Finally, while we have limited our presentation to the case where $\cal A$
is the Kronecker sum of two matrices, the same observations and techniques
apply to the more general case where $\cal A$ is the Kronecker sum of three
or more summands, since this can be reduced to the Kronecker sum
of two matrices. For instance, if 
$$
{\cal A} = M_1\oplus M_2 \oplus M_3 :=
M_1\otimes I\otimes I +
I\otimes M_2\otimes I +
I\otimes I\otimes M_3  ,
$$
we can write
$$
{\cal A} =
M_1\otimes (I\otimes I) +
I\otimes (M_2\otimes I + I\otimes M_3 ) =: M_1 \otimes I + I \otimes {\cal M} ,
$$
and apply the techniques in this paper in a recursive fashion.

\section*{Acknowledgement}
The authors would like to thank Paola Boito for her careful reading 
of the manuscript and helpful
comments.

\section*{Appendix}
In this appendix we prove Lemma \ref{lemma:gamma}.

{\sc Lemma \ref{lemma:gamma}}. {\it For $0<x \le {\alpha} n$ with $ 0 <\alpha < 1$ it holds that
$${\pmb \gamma}(n,x)  \le \frac{1}{1-\alpha}  \frac{x^n}{n} e^{-x}.$$}

\vskip 0.1in
{\it Proof.} We have
\begin{eqnarray*}
{\pmb \gamma}(n,x) 
&=& (n-1)! e^{-x}\left( e^x - \sum_{k=0}^{n-1}  \frac{x^k}{k!}\right) 
= (n-1)! e^{-x}\left( \sum_{k=n}^{\infty}  \frac{x^k}{k!}\right) \\
&=& (n-1)! e^{-x} \frac{x^n}{n!}\left( 1+\sum_{j=1}^{\infty}  \frac{x^{j}}{(n+1) \cdots (n+j)}\right) \\
&\le& e^{-x} \frac{x^n}{n}\left( 1+\sum_{j=1}^{\infty} \alpha^j \frac{n^{j}}{(n+1) \cdots (n+j)}\right) \\
&\le&  e^{-x} \frac{x^n}{n}\left( \sum_{j=0}^{\infty}  \alpha^j\right)  
=  e^{-x} \frac{x^n}{n} \frac{1}{1-\alpha}. \qquad \endproof
\end{eqnarray*}

\bibliography{%
/home/valeria/Bibl/Biblioteca}

\begin{thebibliography}{10}


\bibitem{Bapat10}
{\sc R.~B.~Bapat},
{\em Graphs and Matrices},
Universitext, Springer, London; Hindustan Book Agency, New Dehli, 2010.

\bibitem{Bartels.Stewart.72}
{\sc R.~H.~Bartels and G.~W.~Stewart},
{\em Algorithm 432: Solution of the Matrix Equation $AX+XB=C$},
 Comm.~ACM, 15(9) (1972), pp.~820--826.

\bibitem{Beckermann.11}
{\sc B.~Beckermann},
{\em An error analysis for rational Galerkin projection applied to the
  Sylvester equation},
SIAM J.~Numer.~Anal., 49(6) (2011), pp.~2430--2450.

\bibitem{Beckermann.Reichel.09}
{\sc B.~Beckermann and L.~Reichel},
{\em Error estimation and evaluation of matrix functions via the {Faber}
  transform},
SIAM J.~Numer.~Anal., 47(5) (2009), pp.~3849--3883.

\bibitem{Benzi2013}
{\sc M.~Benzi and C.~Klymko},
{\em Total communicability as a centrality measure}, J.~Complex Networks,
1(2) (2013), pp.~124--149.

\bibitem{Benzi.Simoncini.15tr}
{\sc M.~Benzi and V.~Simoncini},
{\em Decay bounds for functions of matrices with banded or
Kronecker structure}, arXiv:1501.07376 (2015), pp.~1--20.


\bibitem{Druskin.Knizhnerman.98}
{\sc V.~Druskin and L.~Knizhnerman}, {\em Extended {Krylov} subspaces:
  approximation of the matrix square root and related functions}, SIAM
  J.~Matrix Anal.~Appl., 19 (1998), pp.~755--771.

\bibitem{Druskin.Knizhnerman.Simoncini.11}
{\sc V.~Druskin, L.~Knizhnerman, and V.~Simoncini},
{\em Analysis of the rational {K}rylov subspace and {ADI} methods for
  solving the {L}yapunov equation},
SIAM J.~Numer.~Anal., 49 (2011), pp.~1875--1898.

\bibitem{DKZ09}
{\sc V.~Druskin, L.~Knizhnerman, and M.~Zaslavsky}, {\em {Solution of large
  scale evolutionary problems using rational Krylov subspaces with optimized
  shifts}}, SIAM J.~Sci.~Comput., 31 (2009), pp.~3760--3780.

\bibitem{Eiermann2006}
{\sc M.~Eiermann and O.~Ernst}, {\em A restarted {Krylov} subspace method for
  the evaluation of matrix functions}, SIAM J.~Numer.~Anal., 44 (2006),
  pp.~2481--2504.

\bibitem{EHB12}
{\sc E.~Estrada, N.~Hatano,and M.~Benzi},
{\em The physics of communicability in complex networks},
Phys.~Rep., 514 (2012), pp.~89--119.

\bibitem{FrommerGuettelSchweitzer.14}
{\sc A.~Frommer, S.~G{\"u}ttel, and M.~Schweitzer}, {\em Convergence of
  restarted {K}rylov subspace methods for {S}tieltjes functions of matrices},
  SIAM J.~Matrix Anal.~Appl, 35 (2014). pp.~1602--1624.

\bibitem{Frommeretal.14}
\leavevmode\vrule height 2pt depth -1.6pt width 23pt, {\em Efficient and stable
  {A}rnoldi restarts for matrix functions based on quadrature}, SIAM J.~Matrix
  Anal.~Appl., 35 (2014), pp.~661--683.

\bibitem{Frommer2008b}
{\sc A.~Frommer and V.~Simoncini}, {\em Matrix functions}, in Model Order
  Reduction: Theory, Research Aspects and Applications, W.~H.~A. Schilders,
  H.~A. van~der Vorst, and J.~Rommes, eds., Mathematics in Industry, Springer,
  Heidelberg, 2008.

\bibitem{Gradshteyn.Ryzhik.2007}
{\sc I.~S.~Gradshteyn and I.~M.~Ryzhik},
{\em Table of Integrals, Series, and Products},
Seventh Edition, Academic Press, New York, 2007.

\bibitem{Guettel.survey.13}
{\sc S.~G{\"u}ttel}, {\em Rational {K}rylov approximation of matrix functions:
  Numerical methods and optimal pole selection}, GAMM Mitt., 36 (2013),
  pp.~8--31.

\bibitem{Hackbusch}
{\sc W.~Hackbusch},
{\em Numerical tensor calculus}, 
Acta Numerica 2014, pp.~651--742.

\bibitem{ThreeNicks}
{\sc N.~Hale, N.~J.~Higham, and L.~N.~Trefethen},
{\em Computing $A^\alpha$, $\log (A)$, and related matrix functions by contour
integrals},
SIAM J.~Numer.~Anal., 46 (2008), pp.~2505--2523.

\bibitem{HNO06}
{\sc P.~C.~Hansen, J.~G.~Nagy, and D.~P.~O'Leary},
{\em Deblurring Images. Matrices, Spectra, and Filtering},
Society for Industrial and Applied Mathematics, Philadelphia, 2006.

\bibitem{Higham2008}
{\sc N.~J.~Higham}, {\em Matrix Functions -- Theory and Applications}, SIAM,
  Philadelphia, USA, 2008.

\bibitem{Hochbruck1997}
{\sc M.~Hochbruck and C.~Lubich}, {\em On {Krylov} subspace approximations to
  the matrix exponential operator}, SIAM J.~Numer.~Anal., 34 (1997),
  pp.~1911--1925.

\bibitem{Hochbruck1999}
{\sc M.~Hochbruck and C.~Lubich}, {\em {Exponential integrators for
  quantum-classical molecular dynamics}}, BIT, Numerical Mathematics, 39
  (1999), pp.~620--645.

\bibitem{Hochbruck.Ostermann.10}
{\sc M.~Hochbruck and A.~Ostermann}, {\em Exponential integrators}, Acta
  Numerica, 19 (2010), pp.~209--286.

\bibitem{Horn.Johnson.91}
{\sc R.~A.~Horn and C.~R.~Johnson}, {\em {Topics in Matrix Analysis}},
  Cambridge University Press, Cambridge, 1991.

\bibitem{Knizhnerman.91}
{\sc L.~Knizhnerman},
{\em Calculus of functions of unsymmetric matrices using {Arnoldi}'s
  method},
Comput.~Math.~Math.~Phys., 31 (1991), pp.~1--9.

\bibitem{KnizhnermanSimoncini2010}
{\sc L.~Knizhnerman and V.~Simoncini}, {\em {A new investigation of the
  extended Krylov subspace method for matrix function evaluations}}, Numer.~Linear
  Algebra Appl., 17 (2010), pp.~615--638.

\bibitem{Knizhnerman.Simoncini.11}
{\sc L.~Knizhnerman and V.~Simoncini},
{\em Convergence analysis of the Extended Krylov Subspace Method for the
  Lyapunov equation},
Numer.~Math., 118(3) (2011), pp.~567--586.

\bibitem{Kressner.Tobler.10}
{\sc D.~Kressner and C.~Tobler},
{\em Krylov subspace methods for linear systems with tensor product
  structure},
SIAM J.~Matrix Anal.~Appl., 31(4) (2010), pp.~1688--1714.

\bibitem{matlab7}
The MathWorks, Inc.
\newblock {\em M{A}{T}{L}{A}{B} 7}, September 2004.



\bibitem{Moler2003}
{\sc C.~Moler and C.~{Van Loan}}, {\em {Nineteen dubious ways to compute the
  exponential of a matrix, twenty-five years later}}, SIAM Rev., 45 (2003),
  pp.~3--49.

\bibitem{Ng2004}
{\sc M.~K.~Ng},
{\em Iterative Methods for Toeplitz Systems},
Oxford University Press, Oxford, 2004.

\bibitem{Simoncini.survey13}
{\sc V.~Simoncini},
{\em Computational Methods for Linear Matrix Equations},
Technical report, Alma Mater Studiorum - Universit{\`a} di Bologna,
  2013.

\bibitem{Simoncini.Druskin.09}
{\sc V.~Simoncini and V.~Druskin},
{\em Convergence analysis of projection methods for the numerical solution
  of large Lyapunov equations},
 SIAM J.~Numer.~Anal., 47(2) (2009), pp.~828--843.


\bibitem{contest09}
{\sc A.~Taylor and D.~J.~Higham},
\emph{CONTEST: A Controllable Test Matrix Toolbox for MATLAB},
ACM Trans.~Math.~Software, 35 (2009), pp.~26:1--26:17.

\bibitem{Widder.46}
{\sc D.~V.~Widder}, {\em The Laplace Transform}, Princeton University Press,
  1946.

\bibitem{YX08}
{\sc C.~Yang and J.-M.~Xu},
{\em Reliability of interconnection networks modeled by Cartesian 
product digraphs}, Networks, 52 (2008), pp.~202--205.

\end{thebibliography}

\end{document}